\documentclass[11pt,reqno]{amsart}
\usepackage{amsmath, amssymb, graphicx, mathrsfs,  hyperref}
\usepackage{epsfig}
\usepackage{autonum}
\usepackage[makeroom]{cancel}
\usepackage[usenames,dvipsnames,svgnames]{xcolor}
\usepackage[top=1in, bottom=1in, left=1in, right=1in]{geometry} 

\renewcommand{\Re}{\operatorname{Re}}
\renewcommand{\Im}{\operatorname{Im}}
\renewcommand{\min}{\operatorname{min}}
\renewcommand{\max}{\operatorname{max}}
\newcommand{\abs}[1]{\lvert#1\rvert}
\newcommand{\mf}{\mathfrak}
\newcommand{\s}{{\sigma}}

\renewcommand{\a}{\alpha}
\renewcommand{\b}{\beta}

\newcommand{\li}{\text{li}}

\theoremstyle{plain}
\newtheorem{theorem}{Theorem}[section]
\newtheorem{lemma}[theorem]{Lemma}
\newtheorem*{remark*}{Remark}
\newtheorem{corollary}[theorem]{Corollary}

\newtheorem*{example*}{Example}

\begin{document}

	\title[Zeros of partial sums of $L$-functions]{Zeros of partial sums of $L$-functions}
	
	\author[Arindam Roy]{Arindam Roy}
	
	\address{Department of Mathematics and statistics, University of North Carolina at Charlotte, 9201 University City Blvd.
Charlotte, NC 28223}
	
	\email{aroy15@uncc.edu}
	\author[Akshaa Vatwani]{Akshaa Vatwani}
	
	\address{Department of Mathematics,
Indian Institute of Technology Gandhinagar, 
Palaj, Gandhinagar
Gujarat 382355, India}
	
	\email{akshaa.vatwani@iitgn.ac.in }
	

\begin{abstract}
We consider a certain class of multiplicative functions $f: \mathbb N \rightarrow \mathbb C$. Let $F(s)= \sum_{n=1}^\infty f(n)n^{-s}$ be the associated Dirichlet series and  $F_N(s)= \sum_{n\le N} f(n)n^{-s}$ be the truncated Dirichlet series. In this setting, we obtain new Hal\'asz-type results for the logarithmic mean value of $f$.  More precisely, we prove estimates for the sum $\sum_{n=1}^x f(n)/n$ in terms of the size of $|F(1+1/\log x)|$ and show that these estimates are sharp. 
As a consequence of our mean value estimates, we establish non-trivial zero-free regions for  these partial sums $F_N(s)$. 

In particular, we study the zero distribution of  partial sums of the Dedekind zeta function of a number field $K$. 
More precisely, we give some improved results  for the number of zeros upto height $T$  as well as new zero density results for the number of zeros up to height $T$, lying to the right of $\Re(s) =\s$,  where $\s \ge 1/2$. 

\end{abstract}

\maketitle

\thispagestyle{empty}

\section{Introduction}
It is well-known that the zeros of the Riemann zeta function $\zeta(s)$ encode valuable information about the prime numbers. One approach towards understanding the distribution of zeros of $\zeta(s)$ involves the study of partial sums of the Riemann zeta function, defined as 
\begin{equation}\label{eq:zetaX}
\zeta_N(s) = \sum_{n \le N}\frac{1}{n^s}.  
\end{equation}
The investigation of such partial sums is motivated by the approximate functional equation for $\zeta(s)$, given by Hardy and Littlewood \cite{HardyLittlewood}, which approximates $\zeta(s)$ as a sum of two finite Dirichlet series.  More precisely, they showed that  
\begin{equation}\label{eq:approxfn-HL}
\zeta(s) = \sum_{n \le X} \frac{1}{n^s} + \pi^{s-1/2} 
\frac{\Gamma( (1-s)/2 )}{\Gamma(s/2) } \sum_{n \le Y} \frac{1}{n^{1-s}} + O\left( X^{-\s} \right)
+ O\left( Y^{\s-1} |t|^{-\s+1/2} \right), 
\end{equation}
where $s= \s +it$, $0\le \s \le 1$, $X,Y \ge H>0$ and $2\pi XY =|t|$. 
Partial sums thus serve as a good approximation to $\zeta(s)$ and are easier to work with for the purpose of estimation.  
Another interesting feature is the logical relation between the existence of  zeros of $\zeta_N(s) $  and of $ \zeta(s)$ in portions of the half-planes $\s >1$ and $\s > 1/2$, respectively. 
The study of partial sums in this context was initiated by Tur\'an \cite{Turan1948} who showed  that if $\zeta_N(s)$ has no zeros in the half-plane $\s >1+ \frac{(\log N)^3}{\sqrt{N}} $ for all $N$ sufficiently large, then $\zeta(s)$ has no zeros in $\s >1/2$ (i.e. the celebrated Riemann hypothesis is true). In fact, he showed that such a conclusion is true as long as the  number of indices $N \le x$ for which the hypothesis on the Dirichlet polynomial $\zeta_N(s)$ fails is $o(\log x)$ as 
${x \rightarrow \infty}$. In the same paper, he was also able to prove  that $\zeta_N(s)$ does not vanish in the half-plane $\s \ge 1+ 2(\log \log N)/ \log N$, for all large $N$.
In later work, Tur\'an \cite{Turan1960, Turan1959} investigated these connections still further by weakening the hypothesis on zeros of $\zeta_N(s)$ to hold for some half-strip instead of the half-plane, as well as proving implications in the reverse direction.  

It was only in 1983 that the hypotheses required above by Tur\'an on the zeros of $\zeta_N(s)$ were shown to be false by Montgomery \cite{Montgomery1983}, who proved that for any positive constant $c<4/\pi -1$, there exists $N_0=N_0(c)$ such that if $N >N_0$, then $\zeta_N(s)$ does have zeros in the half-plane $\s > 1+\frac{c\log \log N}{\log N}$. In particular, this means that for $N$ sufficiently large, $\zeta_N(s)$ has zeros in $\s >1$. Such results have since been made explicit by the work of Monach \cite{Monach1980}, van de Lune and te Riele \cite{LuneRiele2010},  and Spira \cite{Spira1966, Spira1968}, followed by recent work of  Platt and Trudgian \cite{PlattTrudgian2016} which completes the classification of the finite number of $N$'s for which $\zeta_N$ has no  zeroes in $\s>1$. 
Zeros of $\zeta_N(s)$ have been widely studied and various other results are known in the literature due to a number of mathematicians. 
Among these, we refer the reader to    Borwein, Fee, Ferguson and Waall \cite{BorweinFeeFergusonWaall2007}, Gonek and Ledoan \cite{GonekLedoan2010}, Langer \cite{Langer1931} and Wilder \cite{Wilder1917}.

In a different vein, the limiting value $4/\pi -1$ for the constant $c$ above is now known to be sharp. Indeed, Montgomery and Vaughan \cite[Theorem 4]{MontgomeryVaughan2001} proved the following result. 
\begin{theorem}
\label{thm-MV zero-free}
There is a constant $N_0$ such that if $N>N_0$, then $\zeta_N(s)$ does not vanish in the half-plane 
\[\s \ge 1+ \left( \frac{4}{\pi} -1 \right) \frac{\log \log N}{\log N}. 
\]
\end{theorem}

A  natural question that motivates this paper is whether similar zero-free regions can be exhibited for the partial sums of other $L$-functions. 
Indeed, the study of approximations to $L$-series by Dirichlet polynomials  is  important, yielding  insight into the analytic complexity of general zeta functions. 
For instance, in \cite{BombieriFriedlander1995},  
Bombieri and Friedlander investigated the lengths of the shortest Dirichlet polynomials which can be used as  good approximations to general $L$-functions satisfying certain  conditions. 
This was followed by the work of   Bombieri \cite{Bombieri1997}, who built upon these results to analyze the complexity of computing $\zeta(\s+it)$ in the critical strip, within an error $\epsilon$. As indicated by the work of Hiary \cite{Hiary2011-1, Hiary2011-2, Hiary2014},  approximations by  Dirichlet polynomials are of consequence towards the   development of fast algorithms to compute  $L$-functions at individual points. We also refer the reader to related work by Balanzario \cite{Balanzario1999} as well as Burge\u\i n and Kashin \cite{BurgeinKashin}.

The precise question of interest to us is whether we can construct a family of $L$-functions whose partial sums satisfy properties analogous to those satisfied by $\zeta_N(s)$. In this paper, we answer this question in the affirmative and present the theory of partial sums for a suitable family of $L$-functions. In doing so, we develop new Hal\'asz-type  estimates for the logarithmic mean values of a certain class of multiplicative functions. 

A general mean-value theorem for multiplicative functions taking values in the unit disc was first given by Wirsing \cite{wirsing} and Hal\'asz \cite{halasz, halasz2}. The arguments of Hal\'asz were refined still further by  Montgomery \cite{Montgomery1978}, Elliott \cite{Elliott}, Tenenbaum \cite[chapter III.4]{Tenenbaum}, and have been extended to other classes of multiplicative functions most recently by Tenenbaum \cite{Tenenbaum2017} and Harper, Granville, Soundararajan \cite{GranvilleHarperSound2017}, \cite{GranvilleHarperSound2017-2}.  While  estimates for logarithmic mean values were derived in \cite{MontgomeryVaughan2001} for totally multiplicative functions taking values in the unit disc, it was also  demonstrated that they do not extend to multiplicative functions in general. Thus, working in the right class of multiplicative functions is crucial to our results. We do this as follows. Let $f$ be a multiplicative function, let 
\begin{equation}\label{eq:log of F}
F(s) = \sum_{n=1}^\infty \frac{f(n)}{n^s}, \quad 
-\frac{F'(s)}{F(s)} = \sum_{n=2}^\infty \frac{\Lambda_f(n)}{n^s}, \; 
\text{ and } \log F(s) = \sum_{n=2}^\infty \frac{\Lambda_f(n)}{n^s \log n};    
\end{equation}
 and assume that all these Dirichlet series are absolutely convergent for $\Re(s) >1$. We restrict our attention to the class $\mathcal C(k)$ of  multiplicative functions $f$ which satisfy 
\begin{equation} \label{vonmangoldt}
 |\Lambda_f(n)| \le k \Lambda(n) \text{ for all } n \ge 2,
\end{equation}
where $k$ is some fixed positive constant and $\Lambda(n)$ is the von Mangoldt function. 
We will say that $f$ is $k$-bounded if $f \in \mathcal{ C}(k)$. 
This class of multiplicative functions includes most of the  multiplicative functions studied in the literature and has been an  object of active interest recently. 
For instance, the classical version of Hal\'asz's theorem was generalized to the broader class $\mathcal C(k)$ by Granville, Harper and Soundararajan \cite{GranvilleHarperSound2017}. 
For the  class of $1$-bounded functions, namely $\mathcal C(1)$, recently Drappeau, Granville and Shao \cite{GranvilleShao1, GranvilleShao2, DrappeauGranvilleShao} also  investigated the existence of Bombieri-Vinogradov type remainder estimates.  
In this paper, we develop new Hal\'asz-type  results for the logarithmic mean values of functions in this class.  More precisely, we prove estimates for the sum $\sum_{n\le x} f(n)/n$ in terms of the size of $|F(1+1/\log x)|$ and show that these estimates are sharp. As a consequence, we obtain non-trivial zero-free regions for the partial sums 
\begin{equation} 
	F_N(s) := \sum_{n \le N} \frac{f(n) }{n^s} \label{eq:partialsumF}
\end{equation} of $F(s)$. 

Understanding the distribution of zeros of the partial sums $F_N(s)$ is also an independent area of interest. If $N(T)$ denotes the number of zeros of $F_N(s)$ up to height $T$, then it is known that 
\begin{align}
N(T)= \frac{T}{2\pi} \log M+ O(N),\label{clb}
\end{align}
where $M$ denotes the greatest integer less than or equal to $N$ satisfying $f(M) \ne0$. (See Lemma \ref{gnozt}, Langer \cite{Langer1931}, Wilder \cite{Wilder1917}, Tamarkin \cite{Tamarkin1928}, and Gonek and Ledoan \cite{GonekLedoan2010} for details.) For suitable choices of $f\in \mathcal{C}(k)$, one is able to improve the error term in \eqref{clb}.


Let $O_K$ denote the ring of integers of the field $K$. Consider the Dedekind zeta function of an arbitrary number field $K$, defined for $\Re s >1$ by 
\[
\zeta_K(s) = \sum_{\mf a } \frac{1}{\|\mf a \|^s} = \sum_{n =1}^\infty \frac{a(n)}{n}, 
\]
where the first sum runs over non-zero integral ideals $\mf a$ of $O_K$, $\|\mf a\|$ denotes the absolute norm of $\mf a$, and $a(n)$ denotes the number of integral ideals $\mf a$ with norm $n$. As explained in Section \ref{sec 2}, it can be easily seen that this is an $L$-function associated to a multiplicative function in $\mathcal{ C}(n_0)$, where $n_0 = [K:\mathbb Q]$. 

For the corresponding partial sums 
\begin{align}
\zeta_{K, X} (s)
\mathrel{\mathop:}= \sum_{\|\mathfrak{a}\| \leq X} \frac{1}{\|\mathfrak{a}\|^{s}}
= \sum_{n \leq X} \frac{a (n)}{n^{s}}, \label{eq:parsumze}
\end{align}
it can be shown that the zeros of $\zeta_{K,X}(s)$ lie in some rectilinear strip $\alpha < \s < \beta$. (See Lemma \ref{gnozt}.) 
Letting $\rho_{K,X} =\beta_{K,X}+i\gamma_{K,X}$ denote a typical non-real zero of $\zeta_{K,X}(s)$, one may also consider the number of zeros  up to height $T$, that is, we define 
\begin{align}
N_{K,X}(T) = \#\big\{ \rho_{K,X} : \gamma_{K,X} \in [0,T] \big\}.\label{eq:zeupheT}
\end{align}
In the case that $T$ is the ordinate of a zero, we define $N_{K,X} (T)$ as $\lim_{\epsilon \to 0^{+}} N_{K,X} (T + \epsilon)$.  
Ledoan, the second author and Zaharescu   \cite[Theorem 1]{LedoanRoyZaharescu2014}  obtained  the following  asymptotic formula. 
\begin{theorem}\label{thm:lrz}
	Let $K= \mathbb Q(\zeta_q)$, where $\zeta_q$ is a primitive $q$-th root of unity, $q \ge 2$ and $T, X\ge 3$. Let $M$ be the greatest integer $\le X$ such that $a(M) \ne 0$. Then 
	\[
	N_{K,X}(T) =  \frac{T}{2\pi}\log M + O_q\left(  X \left( \frac{\log \log X}{\log X}\right)^{1-\frac{1}{\phi(q)}}  \right),  
	\]
	where $\phi$ is Euler's totient function. 
\end{theorem}

Another quantity of interest is the number of non-real zeros upto height $T$, lying to the right of $\Re s = \s$. More precisely, we define 
\begin{align}
N_{K,X}(\s, T) = \#\big\{\rho_{K,X} : \gamma_{K,X} \in [0,T], \beta_{K,X}> \s  \big\},\label{eq:zerodense}  
\end{align}
with a similar modification as done for $N_{K,X}(T)$ if $T$ is the ordinate of a zero.
 In the case $K=\mathbb{Q}$, we write $N_{X}(T)$ instead of $N_{\mathbb{Q},X}(T)$ for convenience. The following zero density estimates were obtained by Gonek and Ledoan \cite[Theorem 3]{GonekLedoan2010}. 
\begin{theorem}
\label{thm:GLzerodens}
Let $T\rightarrow \infty$ and suppose that $X =g(T) $ goes to infinity with $T$, such that $g(T) \ll T $. Then, 
\[
N_X(\s, T) \ll TX^{-2(\s-1/2) } (\log T)^5 + (\log T)^2, 
\]
uniformly for $1 \le \s \le 2$. If in addition, we have $g(T) =o(T)$, then 
\[
N_X(\s, T) \ll T \left( \min(X, T/X)   \right)^{-2(\s-1/2) } (\log T)^5,
\]
uniformly for $1/2 \le \s \le 1$. 
\end{theorem}

In this paper, we extend  the result of Theorem \ref{thm:lrz}   still further to Galois extensions $K/\mathbb{Q}$ as well as obtain a  substantial improvement over  the error term, especially for large values of $\phi(q)$. 
Moreover, we also obtain new zero density results such as Theorem \ref{thm:GLzerodens} for partial sums of the Dedekind zeta function of an arbitrary number field $K$. 


\section{Statement of Results} \label{sec 2}
In this section, we  give a brief summary of our results, reordering them here as per convenience of presentation of the proofs. 
Our first result is an asymptotic formula for the number of zeros up to height $T$ of the partial sums  $\zeta_{K,X}(s)$ of the Dedekind zeta function,  where $K$ is  any  Galois extension. This generalizes Theorem \ref{thm:lrz} to the larger family of Galois extensions. Moreover, even for cyclotomic extensions we  improve the error term given in Theorem \ref{thm:lrz} by a factor of $(\log \log X)^{1-1/\phi(q)}$. 
\begin{theorem} \label{thm1}
	Let $K/ \mathbb Q$ be a Galois extension, $[K:\mathbb{Q}] = n_0$ and $T, X \ge 3$. 
	Let $M$ denote the largest integer less than or equal to $X$ such that $a(M)\neq 0$. We have
	\begin{align} \label{lrz1}
	N_{K, X} (T)
	= \frac{T}{2 \pi} \log M + O_{K} \left( \frac{X}{\left(\log X\right)^{1 - \frac{1}{n_0}}}\right),
	\end{align}
	where $N_{K, X} (T)$ is given by \eqref{eq:zeupheT}.
\end{theorem}

We also provide a non-trivial zero density estimate for the partial sums of the Dedekind zeta function $\zeta_K(s)$ where $K$ is an arbitrary number field. 
\begin{theorem} \label{thm2}
	Let $N_{K, X} (\sigma,T)$ be given by \eqref{eq:zerodense}. Let $T\rightarrow \infty$,  $X\ll T$, and let $X$ tend to infinity with $T$. Then for any number  field $K$
	\begin{align}
	N_{K, X} (\sigma,T)\ll_K TX^{-2(\s-1/2)}(\log T)^5+(\log T)^2
	\end{align}
	uniformly for $1\leq \sigma\leq 2$. If in addition, we have  $X=o(T)$, then 
	\begin{align}
	N_{K, X} (\sigma,T)\ll_K T(\min(X, T/X))^{-2(\s-1/2)}(\log T)^5
	\end{align}		
	uniformly for ${1}/{2}\leq \s\leq 1$.
\end{theorem}


We now restrict ourselves to the class $\mathcal C(k)$ of multiplicative functions, defined in the previous section (see \eqref{vonmangoldt}). 
Since by definition, the class $\mathcal C(j)  \subseteq \mathcal C(k)$ whenever $j \le k$, in practice, the index $k$ is chosen to be the smallest positive constant with which we can establish the inequality $|\Lambda_f(n)|\leq k\Lambda(n)$. 
We also note that the $L$-functions associated to  multiplicative functions in the class $\mathcal C(k)$ include many of the most useful $L$-functions. In particular, $\zeta(s)$ is associated to the class $ \mathcal{ C}(1)$, while $\zeta_K(s)$ is associated to $ \mathcal{C}(n_0)$, where $n_0 = [K:\mathbb{Q}]$. The latter claim can be verified using  standard theory of number fields, for instance see the discussion following equations (2.2) and (2.3) of M. Mine \cite{Mine2017}.

For any $f \in \mathcal C(k)$, it can  be shown that zeros of the corresponding partial sums $F_N(s)$, defined in \eqref{eq:partialsumF}, lie within a rectilinear strip of the complex plane given by the inequalities $\alpha <\sigma <\beta$, where $\alpha$ and $\beta$ may depend on $f$ and $N$ (see Lemma \ref{gnozt}). In order to obtain more information about \eqref{eq:partialsumF}, 
our main subject of investigation will be sums of the form 
\begin{equation}\label{eq:F1}
S_1(x) = \sum_{n \le x} \frac{f(n)}{n},  
\end{equation}
as $x \rightarrow \infty$. 
Montgomery and Vaughan \cite{MontgomeryVaughan2001} considered such  sums for totally multiplicative functions $f$ taking values inside the unit disc. 
We generalize these mean value estimates to the class $\mathcal{ C}(k)$.  We obtain the following result, which can  be thought of as a variant of Hal\'asz's classical mean value result (cf. Theorem III.4.7 of \cite{Tenenbaum}), extended to $\mathcal{ C}(k)$. 
	\begin{theorem} \label{thm:F1M1}
	Let $f\in \mathcal C(k)$ and $F(s)$ be the associated Dirichlet series given by \eqref{eq:log of F}. Let $x \ge 3$. For $\a >0$, let
	\begin{align} \label{eq:M1(alpha)}
	M_1(\a):=\left(\sum_{k=-\infty}^{\infty}\max\limits_{\substack{|t-k|\leq \frac{1}{2}\\ \a\leq \s\leq 1}}\left\lvert\frac{F(1+\s+it)}{\s+it}\right\rvert^2\right)^{\frac{1}{2}}.
	\end{align}  
	Then
	\begin{align}
	S_1(x)\ll \frac{k^2(k+1)}{\log x} \int_{1/\log x}^{1}\frac{M_1(\a)}{\a}~d\a, 		
	\end{align}
	where the implied constant is independent of $k$. 
\end{theorem}
{One may express the above bound more effectively by replacing $M_1(\a)$ by a suitable truncation.   
Expressing $M_1(\a)$ given above in terms of $F(\s)$ allows us to obtain the  following more amenable bound.

\begin{theorem} \label{thm:F1estimate}
	Let $f\in \mathcal C(k)$, $F(s)$ be the associated Dirichlet series given by \eqref{eq:log of F} and $x\geq 3$. Then uniformly for  $1+\frac{1}{\log x}<\s\leq 2$, we have  
	\begin{align} 
	S_1(x)\ll_k |F(\s)|(\s-1)^k (\log x)^{k-1}\left((\s-1)^{-4k/\pi}+\log x\right).
	\end{align}	
\end{theorem}
This gives an estimate for the sum $S_1(x)$, for functions whose Dirichlet series $F(s)$ are associated to the class $ \mathcal{ C}(k)$. Notice that $f \in \mathcal{C}(k)$ implies that $f(n) \ll d_k(n)$, where
\begin{align}
d_k(n)=\sum_{n=m_1m_2\dots m_k}1
\end{align}
is the $k$th divisor function. Using this bound on $f$, if we have $F(\s) \ll 1$, versions of the Tauberian theorem 
would only allow us to do slightly better than the trivial estimate $S_1(x) \ll  (\log x)^k$. However, if $F(\s) \ll 1$, and $f$ belongs to the class $ \mathcal{ C}(k)$, then by plugging in $\s =1+ (\log x)^{-\pi /4k}$ into our result above, we obtain the much stronger bound 
\[
S_1(x) \ll (\log x)^{k- \pi /4}.
\] 
Thus, at the cost of more restrictions on $f$, we obtain a power savings over the Tauberian theorem. 

In particular, by taking $\s = 1+ 1/\log x$ in the above result, we have 
\begin{align} \label{eq:sharp}
S_1(x)  \ll |F(1+1/\log x)| (\log x)^{4k/\pi -1}. 
\end{align}
The following example will show us that this bound is sharp. \\
\textbf{Example:}
Let  $g(n)$  be a totally multiplicative function defined by $g(p) = b\left( \frac{1}{2\pi} \log p \right)$, where $b(u)$ is a function having period $1$, given by $b(u) = ie^{i\pi u}$ for $u \in [0,1]$. We proceed to define $f(n)$ by the $k$-fold Dirichlet convolution of $g$, so that  
\begin{equation} \label{choice of F}
F(s) = \sum_{n =1}^\infty \frac{g(n)d_k(n)}{n^s} = 
\prod_{p} \left(  1- \frac{g(p)}{p} \right)^{-k}, \qquad (\Re (s) >1). 
\end{equation}
In order to verify that $f\in \mathcal{ C}(k)$, we observe that the above gives
\begin{align} \label{logF}
\log F(s) &= k \sum_p \sum_{m =1}^\infty \frac{g(p)^m}{mp^{ms}}  \\ 
&= \sum_{n =1}^\infty \frac{\Lambda_f(n)}{n^s \log n}, 
\end{align}
with $\Lambda_f(n) = k (\log p) g(p)^m $ if $n =p^m$, and zero if $n$ is not a prime power.  As $|g(p)| =1$ by construction, this means $|\Lambda_f(n)| \le k \Lambda(n)$, so that $f \in \mathcal C(k)$. 
Consider now the Fourier series for $b(u)$, given by 
\begin{equation}
b(u) = \sum_{r=-\infty} ^\infty a_r e^{2\pi i ru},  \, \text{where }\,  a_r= \int_0^1 b(u)e^{-2\pi i ru} du.
\end{equation}
Plugging this into \eqref{logF} and absorbing the sum with $m\ge 2$ into $O(1)$ gives
\begin{align}
\log F(s) &= k  \sum_p \sum_r \frac{a_r p^{ir}  }{p^ {s }}  + O_k(1)
\\
&= k   \sum_p  \sum_r a_r \sum_{m =1} ^\infty \frac{1  }{mp^{m(s-ir) }}  + O_k(1). 
\end{align}
Noticing that $a_0= -2/\pi$ and $a_1 = 2/\pi$ in our case, we see that as $\s \rightarrow 1^+ $, we have 
\begin{align}
F(\s)  &\sim c_1 (\s - 1)^{2k/\pi},  
\\
F(\s+i) &\sim c_2(\s -1)^{-2k/\pi}, 
\end{align}
for some non-zero complex constants $c_j$. 
Moreover, for this choice of $F(s)$, we have
\[
|S_1(x)| \sim c_3 (\log x)^{\frac{2k}{\pi} -1}, 
\]
showing that the bound \eqref{eq:sharp} is  optimal.

Using the above mean value estimates, we establish that the  zero-free half plane of the form given by Theorem \ref{thm-MV zero-free} holds more generally. More precisely, we are able to derive such a zero-free region for partial sums of $L$-functions associated to the  class $\mathcal{ C}(k)$. 
\begin{theorem} \label{thm:zerofreeF} 
	Let $f\in \mathcal{C}(k)$, $F(s)$ be the associated Dirichlet series given by \eqref{eq:log of F}
and $F_N(s)$ denote its partial sums, given by \eqref{eq:partialsumF}. Then there is a constant $N(k)$ such that if $N>N(k)$, then $F_N(s)\neq 0$ whenever 
	\begin{align}
	\s \ge 1+\left(\frac{4k}{\pi}-1\right)\frac{\log\log N}{\log N}.
	\end{align}
\end{theorem}
We believe that this result is sharp, the details of which we relegate to a forthcoming paper.

Note that  letting $k=1$, we recover Theorem \ref{thm-MV zero-free} for partial sums of the Riemann zeta function. Moroever as a straightforward consequence, we obtain zero-free regions  for the partial sums  $\zeta_{K, X}(s)$ of the Dedekind zeta function of any number field $K$.  

\begin{corollary} \label{cor:zerofreezeta}
	Let $n_0 = [K:\mathbb{Q}]$ and $\zeta_{K,X}(s)$ denote the partial sums given by \eqref{eq:parsumze}. There is a constant $N(n_0)$ such that if $X>N(n_0)$, then $\zeta_{K,X}(s)\neq 0$ whenever 
	\begin{align}
	\s \geq 1+\left(\frac{4n_0}{\pi}-1\right)\frac{\log\log X}{\log X}.
	\end{align}
\end{corollary}

This paper is organized as follows. Sections \ref{sec4} and \ref{sec5} are devoted to the distribution of zeros of  partial sums of the Dedekind zeta function. More precisely, we prove Theorem \ref{thm1} in Section \ref{sec4} and Theorem \ref{thm2} in Section \ref{sec5}. In Sections \ref{sec6} and \ref{sec7}, we develop the theory of Hal\'asz-type estimates for  logarithmic mean values of functions in $\mathcal{ C}(k)$ and prove Theorems \ref{thm:F1M1} and \ref{thm:F1estimate} respectively.  Finally, Section \ref{sec8} contains the proof of Theorem \ref{thm:zerofreeF}, giving  new zero-free regions for the relevant families of partial sums.

\section{Preliminary results} 
In this section, we give a summary of some preliminary results that we will need. 
We will  use the variant of the classical Tauberian theorem given by Theorem 2.4.1 of \cite{sievebook}. This was proved by Raikov \cite{Raikov1938} in 1938. 
\begin{lemma}\label{thm taub}
	Let  
	\begin{align}
	H(s) = \sum_{n\ge 1} \frac{a_n}{n^s}
	\end{align}
	be a Dirichlet series with non-negative coefficients converging for $\Re(s) >1$. Suppose that $H(s)$ extends analytically at all points on $\Re(s)=1$ apart from $s = 1$, and 
	that at $s=1$ we can write 
	\begin{align}
H(s)= \frac{h(s)}{(s-1)^{1-\alpha}}, 
	\end{align}
	for some $\alpha \in \mathbb{R}$ and some $h(s)$ holomorphic in the region  Re$(s) \ge 1$ and non-zero there. Then as $x \rightarrow \infty$, 
	\begin{align}
	\sum_{n \le x} a_n \sim \frac{cx}{(\log x)^\alpha}, 
	\end{align}
	with $c = h(1) / \Gamma(1-\alpha)$, where $\Gamma$ is the usual Gamma function. 
\end{lemma}

We will also use the following mean value result for Dirichlet polynomials, due to Montgomery and Vaughan \cite{MontgomeryVaughan1974}.
\begin{lemma}\label{mvt}
	If $\displaystyle\sum_{n=1}^{\infty}n|a_n|^2$ converges, then 
	\begin{align}
	\int_0^{T}\left\lvert\sum_{n=1}^{\infty}a_nn^{-it}\right\rvert^2~dt=\sum_{n=1}^{\infty}|a_n|^2(T+O(n)).
	\end{align}
\end{lemma}
The following lemma from \cite[p.~213]{titchmarsh} will be used to bound the argument of an analytic function. 
\begin{lemma}\label{tbook}
	Let $F(s)$ be an analytic function except for a pole at $s=1$ and be real for real $s$. Let $0\leq a<b<2$. Suppose that $T$ is not an ordinate of any zero of $F(s)$. Let $|F(\sigma+it)|\leq M$ for $\sigma\geq a$, $1\leq t\leq T+2$ and $\Re (F(2+it))\geq c>0$ for some $c\in \mathbb{R}.$ Then, 	for $\sigma\geq b$, 
	\begin{align}
	|\arg F(\sigma+iT)|\leq \frac{c}{\log\frac{2-a}{2-b}}\left(\log M+\log \frac{1}{c}\right)+\frac{3\pi}{2}.
	\end{align}
\end{lemma}

\section{Proof of Theorem \ref{thm1}} \label{sec4}

The proof of Theorem \ref{thm1} relies on the following auxiliary lemma, which is a significant generalization and improvement of Lemma 3 of \cite{LedoanRoyZaharescu2014}. Our key idea here is to translate the condition $a(n) \ne 0$ into a Chebotarev condition and then apply the Tauberian theorem  \ref{thm taub}. 
\begin{lemma} \label{l2}
	Let $K/ \mathbb{Q}$ be a Galois  extension and $[K:\mathbb{Q}]= n_0$. Let $a(n)$ denote the number of integral ideals of ring of integers $O_K$ with norm $n$.
	Then
	\begin{align}
	A(x) : = \#\{n \leq x \colon a (n) \neq 0\}
	\le (1+o(1)) C_K \left( \frac{x}{(\log x)^{1 -1/n_0 }} \right), 
	\end{align}
	where $C_K$ is a constant depending on the field $K$.  
\end{lemma}
\begin{proof}
	Recall that a subset $\mathcal{P}$ of the set of the rational primes $\mathbb{P}$ is called a Chebotarev set if there is a Galois extension $K/\mathbb{Q}$ of number fields with Galois group $G$ and absolute discriminant $d_{K}$ such that
	\begin{align}
	\mathcal{P}=\left\{p\in\Bbb{P}\mid p\text{ is unramified with }\left(\frac{K/\Bbb{Q}}{p}\right)\subseteq C\right\}.
	\end{align}
	Here, for $p$ unramified (or equivalently, $p\nmid d_{K}$), $\left(\frac{K/\Bbb{Q}}{p}\right)$ denotes the Artin symbol at $p$, and $C$ is a union of conjugacy classes of $G$. 
	We also note that orthogonality relations for Artin $L$-functions give us 
	\begin{align} 
	\sum_{p \in \mathcal{P}} \frac{1}{p^s}  = \frac{|C|}{|G|} 
	\log\left( \frac{1}{s-1}\right) + \theta(s), \label{orthog}
	\end{align}
	where $\theta(s)$ is analytic for Re$(s) \ge 1$. 
	
	Let $n_0 = [K : \mathbb Q]$. 	
	It is known 
	that for all but finitely many primes $p$ (in fact for all unramified primes), we have 
\begin{align}
	a(p) = \left\lbrace \begin{array}{ll}
n_0	& \text{ if $p$ splits completely }   \\ 
0	& \text{ otherwise. } 
	\end{array} \right.
\end{align} 
	Since the Frobenius element Frob$_{\mathfrak p} = $ id for any prime ideal $\mathfrak p$ lying above $p$, we have that $p$ splits completely if and only if the Artin symbol $ \left(\frac{K/\Bbb{Q}}{p}\right) =$ id. 
	The set of primes $p$ that split completely in $K$ is thus a Chebotarev set, which we may denote as $\mathcal P' (= \mathcal{P'}(K))$ for the remainder of this proof. 
	
	Any integer $n$ can be uniquely written as $uv^2$, where $v^2$ denotes the squarefull part of $n$ and $u$ is squarefree. Then a necessary condition for $a(n) = a(u)a(v^2)$ to be non-zero is that $a(u) \neq 0$, that is,  $ u$ must be composed only of primes ${p \in \mathcal{P}'} $. With this in mind, we first estimate 
	\begin{align}
	B(x) := \# \{ u \le x  : a(u) \ne 0, \mu^2(u) \ne 0 \} = 
	\sum_{\substack{ u \le x \\  u| \left( \prod_{p \in \mathcal{P} } p \right)}   }1.  
	\end{align}
	Let $b(u) = 1$ if $u | \prod_{p \in \mathcal{P} } p$, and zero otherwise. Clearly $B(x) = \sum_{u \le x} b(u)$ and we have the formal Dirichlet series
	\begin{align}
	F(s) := \sum_{u =1}^\infty \frac{b(u)}{u^s} = \prod_{p \in \mathcal{P} } \left( 1+ \frac{1}{p^s}\right). 
	\end{align} 
	Using \eqref{orthog}, we have, 
	\begin{align}
	\log F(s) &= \sum_{p \in \mathcal{P}} 
	\log\left( 1+ \frac{1}{p^s}\right) 
	= \sum_{p \in \mathcal{P}} \sum_{ m=1}^\infty (-1)^{m+1} \frac{1}{mp^{ms}}= \frac{1}{n_0} \log\left( \frac{1}{s-1} \right)   + \theta'(s), 
	\end{align}
	for some function $\theta'(s)$ which is analytic for $\Re(s)\ge 1$. 
	This allows us to write 
	\begin{align}
	F(s) = \frac{h(s)}{(s-1)^{1/n_0}}, 
	\end{align}
	where $h(s)$ is analytic and non-zero for Re$(s) \ge 1$. Applying  Lemma \ref{thm taub}, we see that 
	\begin{align}
	B(x) =(1+o(1))  \frac{cx}{ (\log x)^{1- \frac{1}{n_0}}}, 
	\end{align}
	for some constant $c$ depending on $K$. 
	
	Coming back to $A(x)$, by our earlier discussion we have 
	\begin{align}
	A(x) = \sum_{\substack{uv^2\le x \\ a(u)a(v^2) \ne 0 }} 1 
	&\le \sum_{v \le \sqrt x} \sum_{\substack {u \le x/v^2 \\ a(u) \ne 0}} 1 = \sum_{v \le \sqrt x}  B(x/v^2).  
	\end{align}
	Let $1 < z < \sqrt x$, to be chosen later. 
	We split the final sum into two sums $A_1(x)$ and $A_2(x)$ (say), ranging over $v \le z$ and $ z < v \le \sqrt x$ respectively. 
	Then the trivial bound  $B(x) \ll x$ gives 
	\begin{align}
	A_2(x) = \sum_{z< v \le \sqrt x } B(x/v^2) 
	\ll x \sum_{ v>z} \frac{1}{v^2} \ll \frac{x}{z}. 
	\end{align}
	Choosing $z = \log x$ gives $A_2(x) \ll x /\log x$. 
	For the sum $A_1(x)$, we notice that in the range $1\le v \le \log x$, $\log (x/v^2) \sim  \log x$, so that
	\begin{align}
	B(x/v^2) = (1+o(1)) \frac{cx}{v^2 (\log x)^{1- \frac{1}{n_0}} }, 
	\end{align}
	 in this range.  
	This gives 
	\begin{align}
	A_1(x) =   \sum_{v \le (\log x) } B(x/v^2) 
	&= (1+o(1)) \frac{cx}{(\log x)^{1- \frac{1}{n_0}} }  \sum_{v \le (\log x) }
	\frac{1}{v^2}\\
	&= (1+o(1))  \frac{C_K x}{(\log x)^{1- \frac{1}{n_0}} },  \label{asymp for A_1}
	\end{align}
	for some  constant $C_K$, since the  series $\sum_{v=1}^\infty 1/v^2$ is convergent. 
Since $A_2(x) = o(A_1(x))$ with our choice of $z$, we have obtained 
	\begin{align}
	A(x) \le A_1(x) + A_2(x)  = (1+o(1)) A_1(x). 
	\end{align}
 Plugging in \eqref{asymp for A_1} completes the proof of the lemma.  
\end{proof}

The following lemma is instrumental to the proof of Theorem \ref{thm1}. Similar kind of results have appeared in the work of Langer \cite{Langer1931}, Wilder \cite{Wilder1917}, and Tamarkin \cite{Tamarkin1928}.
\begin{lemma}\label{gnozt}
	Let $f$ be a real valued arithmetic function with $f(1)=1$. Let
	\begin{align}
		F_N(s)=\sum_{1\leq n\leq N} \frac{f(n)}{n^s}.
		\end{align}
		 Then there exists a  {positive} real constant $B$ such that all the zeros of $F_N$ lie in the strip $|\Re(s)|<B$. Moreover, if $N(T)$ denotes the number of zeros of $F_N$ in the region $|\Re(s)|\le B$, $0\leq \Im(s)\leq T$, then 
		 \begin{align}
			 N(T)=\frac{T}{2\pi}\log M+O(E(N)) ,\label{eqN}
			 \end{align}
			 where $M$ is  the largest integer less than or equal to $N$ such that  $f(M)\neq 0$, and   
			 \begin{align}
				E(N) := \#\{n\leq N: f(n)\neq 0\} \label{nzc}.
				 \end{align}
\end{lemma}
\begin{proof}

One can always find a real number $B$ such that 
	\begin{align}
		\sum_{2\leq n\leq M}\frac{|f(n)|}{n^B}<1\label{rb}
		\end{align}
		and 
		\begin{align}
			\sum_{1\leq n\leq M-1}|f(n)|n^B<|f(M)|M^B.\label{lb}
			\end{align}
Using the reverse triangle inequality 
\[ |F_N(s)| \le 1-  \bigg| \sum_{2\leq n\leq M}\frac{f(n)}{n^s}\bigg|, 
\]
followed by \eqref{rb}, we see that  $|F_N(s)|>0$ for $\Re(s)>B$. Similarly, using \eqref{lb}, we can deduce the same for $\Re(s) < -B$. 
Thus, we have $|F_N(s)|>0$ for $|\Re(s)|>B$,  justifying the first claim of the lemma. 
	
Let $R$ denote the rectangle bounded by the lines $|\Re(s)|=B$, $ \Im(s)=0$ and $\Im(s)= T$. For the sake of simplicity, we assume that there are no zeros of $F_N$ on the boundary of $R$. Then one has
	\begin{align}
	N(T)
	  = \frac{1}{2 \pi i} \int_{R} \frac{F_N' (s)}{F_N(s)} \,ds.
	\end{align}
	If  $\triangle_{C} \arg F_N(s)$ denotes the change in $\arg F_N (s)$ as $s$ traverses  along the path $C$ in the positive sense, then we have
	\begin{align} \label{eq19}
	2 \pi N (T)
	 = \int_{R} \mbox{Im} \left(\frac{F_N' (s)}{F_N (s)}\right) \,ds
	 = \triangle_{R} \arg F_N (s).
	\end{align}
For $\Re (s)=B$, we observe from \eqref{rb} that 
	\begin{align}
	\abs{F_N (s) -1}
	 <1.
	\end{align}
	It follows that $\Re F_N (B + i t) > 0$ for $0 \leq t \leq T$. This means that $\arg F_N(B+it)$ is uniformly bounded, that is, 
	\begin{align} \label{eq21}
	\triangle_{[0, T]} \arg F_N (B + i t)\ll 1, 
	\end{align}
	Next, we decompose $F_N(s)$ into its real part and its imaginary part. We have
	\begin{align}
	F_N (s)= \sum_{n \leq N} f (n) \exp \{{-(\sigma + i t) \log n}\} = \sum_{n \leq N} \frac{f (n) [\cos (t \log n) - i \sin (t \log n)]}{n^{\sigma}},
	\end{align}
	so that
	\begin{align}
	\Im (F_N(\sigma + i T))
	 = - \sum_{n \leq N} \frac{f (n) \sin (T \log n)}{n^{\sigma}}.
	\end{align}
	By a generalization of Descartes's Rule of Signs (see P\'{o}lya and Szeg\"{o} \cite{PolyaSzego1971}, Part V, Chapter 1, No. 77), the number of real zeros of $\Im (F_N(\sigma + i T))$ in the interval $-B \leq \sigma \leq B$ is less than or equal to the number of nonzero coefficients $f(n) \sin (T \log n)$. By \eqref{nzc}, the number of nonzero coefficients $f (n)$ is $E(N)$.
Since the change in argument of $F_N(\sigma + i T)$ between two consecutive zeros of  $\Im (F_N(\sigma + i T))$ is at most $\pi$, it follows that
	\begin{align} \label{eq22}
	  \triangle_{[-B,B]} \arg F_N (\sigma + i T)\ll E(N).
	\end{align}
	Similarly, 
	\begin{align} \label{eq23}
	  \triangle_{[-B,B]} \arg F_N (\sigma )\ll E(N).
	\end{align}
	
 To compute the change of argument of $F_N$ along the left edge of $R$, we write
	\begin{align}
	F_N(-B+ i t)
	 = \left[1 + \sum_{1\leq n\leq M-1}\frac{f (n) n^{B - i t}} {f( M) M^{B- i t}}\right]f(M)M^{B - i t}.
	\end{align}
Now,
	\begin{align} \label{eq24}
	\begin{split}
	  \triangle_{[0, T]} \arg F_N(-B+ i t) 
	 &= \triangle_{[0, T]} \arg \left[1 + \sum_{1\leq n\leq M-1}\frac{f (n) n^{B - i t}} {f( M) M^{B- i t}}\right]+ \triangle_{[0, T]} \arg f(M)M^{B - i t}.
	\end{split}
	\end{align}
From \eqref{lb} one has 
	\begin{align} 
	\left\lvert   \sum_{1\leq n\leq M-1}\frac{f (n) n^{B - i t}} {f( M) M^{B- i t}} \right\rvert
	 < 1,
	\end{align}
	and hence as done before, we may conclude that 
	\begin{align} \label{eq25}
	\triangle_{[0, T]} \arg \left[1 + \sum_{1\leq n\leq M-1}\frac{f (n) n^{B - i t}} {f( M) M^{B- i t}}\right]\ll 1.
	\end{align}
We compute the second term in \eqref{eq24} as follows
	\begin{align} \label{eq26}
	\begin{split}
	\triangle_{[0, T]} \arg f (M) M^{B-i t}
	&=\triangle_{[0, T]} \arg f (M) M^{B} \exp\{-i t \log M\}\\
	&=\triangle_{[0, T]} \arg \exp\{-i t \log M\}\\
	 &=- T \log M.
	\end{split}
	\end{align}
	Then substituting \eqref{eq25} and \eqref{eq26} into \eqref{eq24}, we obtain
	\begin{align} \label{eq26'}
	  \triangle_{[0, T]} \arg F_N(-B + i t) = -T \log M + O (1).
	\end{align}
	We may now substitute \eqref{eq21}, \eqref{eq22}, \eqref{eq23}, \eqref{eq26'} into \eqref{eq19} to obtain Lemma \ref{gnozt}.
	\end{proof}
\subsection{Proof of Theorem \ref{thm1}}
We substitute $N=\lfloor X\rfloor$ and $f(n)=a(n)$ in Lemma \ref{gnozt}. Then from Lemma \ref{l2} we have
\begin{align}
	E(N)\ll_K  \frac{X}{(\log X)^{1 -1/n_0 }}.\label{eb}
	\end{align}
Finally, by substituting \eqref{eb} in \eqref{eqN} we conclude the proof of Theorem \ref{thm1}.
\section{Proof of Theorem \ref{thm2}} \label{sec5}
\subsection{Preliminary results}
In our proof of Theorem \ref{thm2}, we will need to use  properties of certain arithmetical functions on number fields. We proceed to set this up below. Since ideals $\mf a$ in $O_K$ factor uniquely, we may define the M\"obius function on $O_K$ as follows. 
\begin{align}
\mu(\mf a) =  \left\lbrace \begin{array}{cl}
1 &  \text{ if } \mf a  = O_k, \\ 
0 &  \text{ if  } \mf a \text{ is not squarefree}, \\
(-1)^m & \text{ if  } \mf a \text{ is the product of $m$ distinct prime ideals}.  
\end{array} 
\right. 
\end{align}
Then the fundamental property of the M\"obius function follows in the same way as in the classical case. We have 
\begin{align} \label{fund prop}
\sum_{ \mf a \subseteq \mf d  } \mu(\mf d) = \left\lbrace 
\begin{array}{cl}
1 &  \text{ if } \mf a = O_k \\
0  & \text{ otherwise,}
\end{array}
\right.  
\end{align}
where the  sum runs over all integral ideals $\mf d$ containing $\mf a$. 

We will also need to consider the number field analogue of the divisor function, given   on $O_K$ by
\begin{align}
d(\mf m ) := \sum_{ \mf a \mf b = \mf m  } 1. 
\end{align}

By comparing the coefficients of the expressions 
\begin{align}
\zeta_K(s)^2 = \sum_{ \mf m } \frac{d(\mf m)}{\| \mf m\|^s} = \sum_{m=1}^\infty \frac{1}{m^s} \sum_{rn =m }a(r)a(n), 
\end{align}
where $a(n)$ denotes the number of integral ideals with absolute norm $n$,  we see that 
\begin{align}\label{div estimate}
\sum_{\mf m : \|\mf m\| \le x} d(\mf m) &= \sum_{m \le x} 
\sum_{rn =m }a(r)a(n) 
\\
&= \sum_{r \le x} a(r) \sum_{n \le x/r} a(n) 
\\
&\ll_K  x(\log x).
\end{align}
Here, we have used the well known estimate $\sum_{n \le x} a(n) \sim c_Kx$, where $c_K$ is the residue of $\zeta_K(s)$ at its simple pole $s=1$.  While this estimate is immediate upon applying the Tauberian theorem, the interested reader may also refer to the work of Landau \cite{Landau}. 

We will also need a bound on the mean square of this divisor function. Let us first note that 
\begin{align}
 d(\mf m)^2 &= \sum_{\mf a \mf b =\mf m}  d(\mf m) 
 =\sum_{\mf a \mf b =\mf m}  d(\mf a \mf b ) 
\le \sum_{\mf a \mf b =\mf m}  d(\mf a) d(\mf b ).  
\end{align}
With this in hand, we can use an argument similar to that used to treat the sum  \eqref{div estimate}, to obtain  
\begin{align}\label{div sq estimate}
\sum_{\mf m: \|\mf m\| \le x }  d(\mf m)^2  \ll_Kx(\log x)^3. 
\end{align}

We now proceed to give the proof of Theorem \ref{thm2} below. 

\subsection{Proof of Theorem \ref{thm2}}
 We define 
\begin{align}
	f_K(s) =  \zeta_{K, X}(s)M_{K,Y}(s)-1, 
\end{align}
where 
\begin{align}
	M_{K,Y}(s) :=\sum_{ \|\mathfrak{b}\| \leq  Y } 
	\frac{\mu(\mathfrak b )}{\|\mathfrak{b}\|^{s}}.
\end{align}
$Y$ will later be chosen as a function of $X$ and $ T$ that tends to infinity as they do. 
This gives 
\begin{align}
	f_K(s) &= \sum_{ \|\mathfrak{a}\| \leq  X } 
	\frac{1}{ \| \mf{a} \|^s  } \sum_{ \|\mathfrak{b}\| \leq  Y } 
	\frac{\mu(\mathfrak b )}{\|\mathfrak{b}\|^{s}} -1 
	\\
	&= \sum_{ \|\mathfrak{m}\| \leq  XY } 
	\frac{g( \mf m)}{ \|\mf m \|^s  }    \qquad \text{(say)},  
\end{align} 
where 
\begin{align} \label{expr for a(m)}
g( \mf m ) = \sum_{\substack{ \mf a \mf b = \mf m 
		\\ \| \mf a\| \le X \\ \|\mf b \| \le Y  }} 
	\mu(\mf b),  \qquad  (\mf m \neq O_k,  \|\mf m \| \le XY) 
\end{align}
Letting $\mf m = 1$ denote the ideal $O_K$ for convenience, we see that  $g(1) = 0$. Setting $Z = \min{(X,Y) }$, we see that \eqref{expr for a(m)} is simply the full sum over ideals $\mf b$ containing $\mf m$ whenever $\|\mf m\| \le Z$. By the  fundamental property \eqref{fund prop} of the M\"obius function, we have  $g( \mf m ) = 0 $ for $\|\mf m \| \le Z$. Hence, we may write 
\begin{align}\label{fkd}
f_K(s) = \sum_{Z < \|\mathfrak{m}\| \leq  XY } 
\frac{g( \mf m)}{ \|\mf m \|^s  }. 
\end{align}
Clearly, \eqref{expr for a(m)} implies that 
\begin{align} \label{bound for g}
|g(\mf m)| \le d(\mf m).   
\end{align}  
We set 
\begin{align} \label{fk}
h(s) = 1- f_K(s)^2 = 
\zeta_{K, X}(s) M_{K,Y}(s) \big( 2-\zeta_{K, X}(s) M_{K,Y}(s)   \big). 
\end{align}
Then $h_K(s)$ is holomorphic and vanishes at the zeros of Let $\s\geq 2$. Then, from \eqref{div estimate} and \eqref{fk}
\begin{align}\label{hk}
f_K(s)&\ll \sum_{Z < \|\mathfrak{m}\| \leq  XY } 
\frac{d( \mf m)}{ \|\mf m \|^2  }\ll \frac{\log Z}{Z}\leq \frac{1}{2}
\end{align}
for $Z$ sufficiently large.  Combining this with \eqref{fk}, we find that  $h_K(2+it)\neq 0$ for $Z$ sufficiently large.
Let $\nu(\s',T)$ denote the number of zeros of $h_K(s)$ in the rectangle $\s>\s'$ and $0<t\leq T$.  We apply Littlewood's lemma concerning the zeros of an analytic function in a rectangle (see Titchmarsh \cite[p.~221]{titchmarsh}). This gives 
	\begin{align} 
2\pi\int_{\s_0}^2 \nu\left(\s, T\right)~d\s &= \int_{0}^T\log |h_K(\s_0+it)|~ dt -\int_{0}^T\log |h_K(2+it)|~dt \\
&+\int_{\s_0}^2 \arg h_K(\s+iT)~d\s -\int_{\s_0}^2\arg h_K(\s)~ d\s,\label{HL}
\end{align}
where $\s_0\geq \frac{1}{2}$ is fixed. From \eqref{fk} and \eqref{hk}, we deduce that 
\begin{align}
\Re(h_K(2+it))\geq \frac{1}{2}
\end{align}
for large $Z$. From \eqref{fkd} and \eqref{fk} we have 
\begin{align}
h_K(s)\ll\left(\sum_{Z < \|\mathfrak{m}\| \leq  XY } 
\frac{d( \mf m)}{ \|\mf m \|^{\s}}\right)^2\leq (XY)^{2(1-\s)}(\log XY)^{2}  
\end{align}
for $\s<1$.  Therefore, for $0<\a<\frac{1}{2}$ and $\s>\a$
	\begin{align}
h_K(\s+it)\ll (XY)^2 (\log XY)^{2}.
\end{align}
Also $h_K(2+it)\geq \frac{1}{2}$ for large $Z$. Therefore, from Lemma \ref{tbook}, we have 
\begin{align}
\arg h_K(\s+iT)-\arg h_K\left(\s+i\frac{T}{2}\right)\ll \log XY
\end{align} 
for $\s\geq \s_0$. 
	This gives 
\begin{align}
\int_{\s_0}^2 \arg h_K(\s+iT) ~d\s - \int_{\s_0}^2\arg h_K\left(\s\right) ~d\s\ll\log XY\ll \log T\label{arb}
\end{align}
for $\s>\s_0$.

Since $h_K(s)$ is analytic for $\s\geq 2$ and $h_K(s)\to 1$ as $\s\to \infty$, by the residue theorem 
\begin{align}
\int_{0}^T\log h_K(2+it)~dt=\int_{2}^{\infty}\log h_K\left(\s\right)~d\s-\int_{2}^{\infty}\log h_K\left(\s+iT\right)~d\s.\label{vth}
\end{align}
Also,
\begin{align}
\log |h_K(s)|\leq \log \left(1+|f_K(s)|^2\right)\leq |f_K(s)|^2 \label{hfx} 
\end{align}
and 
\begin{align}
\log |h_K(s)|=\Re (\log h_K(s)).
\end{align}
Using this along with \eqref{fkd}, \eqref{hk} and \eqref{hfx} we have 
\begin{align}
-\int_{0}^T\log |h_K(2+it)|~dt&\ll \int_{2}^{\infty} |f_K(\s)|^2~d\s\ll 1. \label{rvb}
\end{align}

Thus, it remains to estimate only the first integral in \eqref{HL}, which we do by using Lemma \ref{mvt}.
Let $\s_0>\frac{1}{2}$. Then from \eqref{fkd} and \eqref{fk} we have 
\begin{align}
\int_{0}^T\log h_K(\s_0+it)~dt&\ll \int_{0}^T\left\lvert \sum_{Z < \|\mathfrak{m}\| \leq  XY } 
\frac{g( \mf m)}{ \|\mf m \|^{\s_0+it}  }\right\rvert^2~dt\\
&\ll\sum_{Z < \|\mathfrak{m}\|\leq  XY } 
\frac{d^2( \mf m)}{  \|\mathfrak{m}\|^{2\s_0}  }\left(T+O(\|\mathfrak{m}\|)\right)
\\
&\ll (T+XY)(XY)^{1-2\s_0}(\log XY)^{4}+(T+Z)Z^{1-2\s_0}(\log Z)^{4}
\\
&\ll TZ^{1-2\s_0}(\log T)^{4}+(XY)^{2-2\s_0}(\log T)^{4}+Z^{2-2\s_0}(\log T)^{4}. 
\end{align}
 Here, in the penultimate step, we have used the estimate \eqref{div sq estimate}. 
Thus, we have 
	\begin{align}
\int_{\s_0}^2 \nu\left(\s,T\right)~d\s\ll Tz^{1-2\s_0}(\log T)^{4}+&(XY)^{2-2\s_0}(\log T)^{4}+z^{2-2\s_0}(\log T)^{4}
 +\log T,\label{lvb}
\end{align}
for large $Z$. 
	For $1<\s_0\leq 2$, we take $X=Y=Z$ and we find that 
\begin{align}
\int_{\s_0}^2 \nu\left(\s, T\right)~d\s\ll (T+X)X^{1-2\s_0}(\log T)^{4}+\log T,\label{lvb1}
\end{align}
Combining \eqref{HL}, \eqref{arb}, \eqref{rvb}, \eqref{lvb}, and using the inequality 
\begin{align}
\int_{\s_0}^2 \nu\left(\s, T\right)~d\s\geq \int_{\s_0}^1 N\left(\s, T\right)~d\s,\label{gn}
\end{align}
which follows from \eqref{fk}, we obtain
\begin{align}
\int_{\s_0}^2 N\left(\s, T\right)~d\s\ll (T+X)X^{1-2\s_0}(\log T)^{4}+\log T.\label{las1}
\end{align}
uniformly for $1<\s_0\leq 2$.  For $\frac{1}{2}-\frac{1}{\log T}<\s_0\leq 1$ and $X\leq \sqrt{T}$, we take $X=Y=Z$. Then from \eqref{lvb} and \eqref{gn} one finds that 
\begin{align}
\int_{\s_0}^2 N\left(\s, T\right)~d\s\ll TX^{1-2\s_0}(\log T)^{4}.\label{las2}
\end{align}
Lastly, if $\sqrt{T}<X=o(T)$, we take $XY=T$. Then from the right side of \eqref{lvb} and \eqref{gn} we have 
\begin{align}
\int_{\s_0}^2 N\left(\s, T\right)~d\s\ll T(T/X)^{1-2\s_0}(\log T)^{4}.\label{las3}
\end{align}
Now for $\s_0>\frac{1}{2}-\frac{1}{\log T}$ we have 
\begin{align}
N(\s,T)\leq \log T\int_{\s_0}^2N(\s, T)~dt.\label{las4}
\end{align}
Hence the theorem follows by combining \eqref{las1}, \eqref{las2}, \eqref{las3} and \eqref{las4}.

\section{Mean Value Estimates: Proof of Theorem \ref{thm:F1M1}} \label{sec6}
Although we build upon the proof of Theorem 2 in \cite{MontgomeryVaughan2001},  our presentation is self-contained for the sake of completeness. 
We begin by noting  that $M_1(\alpha) \gg 1$ uniformly for $0< \alpha \le 1$. In order to show this, it suffices to observe that 
\begin{align}
\big|\log |F(2+it)| \big| &\le \big|\log (F(2+it))\big| = \bigg|  \sum_{n=1}^\infty \frac{\Lambda_F(n)}{n^{2+it} \log n} \bigg| 
\le k \sum_{n =1}^\infty \frac{\Lambda(n)}{n^2 \log n}
\end{align}
is bounded uniformly in $t$. This means that $|F(2+it)|$ must be greater than some absolute positive constant $c$, for all $t \in \mathbb{R}$.  Then, by the definition \eqref{eq:M1(alpha)} of $M_1(\alpha)$, we have 
\begin{align}
M_1(\alpha) \ge \left(\sum_{k=-\infty}^{\infty} \max\limits_{\substack{|t-k|\leq \frac{1}{2} }} \left\lvert\frac{F(2+it)}{1+it}\right\rvert^2\right)^{\frac{1}{2}} \gg c \left(\sum_{k=-\infty}^{\infty} 
\frac{1}{k^2}
\right)^{\frac{1}{2}} \gg 1,
\end{align} 
uniformly for $0 < \alpha \le 1$. 
Recalling that we want to show 
\begin{align} \label{eq:want}
F_1(x) \log x \ll_k \int_{1/\log x}^{1} \frac{M_1(\alpha)}{\alpha} d\alpha, 
\end{align}
we now aim to prove this for $x\ge x_0$, adjusting the implicit constant to deal with the range $3 \le x < x_0$. 

The function $|F_1(x)| \log x$ is an increasing function on each interval $I_n= [n, n+1)$. Thinking of the right hand side of \eqref{eq:want} (up to constant) as $V(x)$,  we want to show that $|F_1(x)| \log x \le V(x)$, where $V(x)$ is a given  function increasing with respect to $x$. If we know that this is true for some $x=x_0$ (say) in some interval $I_{n_0}$, then on any interval to the right of $I_{n_0}$, the inequality automatically follows for those $x$ such that $ |F_1(x_0)| \log x_0 \ge |F_1(x)| \log x$. Using such arguments,    we only need to prove the inequality for those $x \ge x_0$ which satisfy 
\begin{equation}\label{eq:condition on x}
  x_0 \le u \le x \implies    |F_1(u)| \log u < |F_1(x)| \log x. \end{equation}
 
With all this in hand, we now proceed to prove the result.  
The identity 
\[
\log x = \log n + \frac{(\log n) \log (x/n) }{ \log x} + \frac{(\log (x/n))^2}{\log x}
\]
gives us 
\begin{align}
	F_1(x)  \log x &= \sum_{n \le x} \frac{f(n)}{n} \log n  + \frac{1}{\log x} \sum_{n \le x} \frac{f(n)}{n} (\log n) \log (x/n) + \frac{1}{\log x} \sum_{n \le x} \frac{f(n)}{n} (\log (x/n))^2 
	\\
	&= T_1(x)+T_2(x)+T_3(x)   \qquad \text{ (say)}
\end{align}

\subsection{Estimate for $T_1(x)$}
We first consider the sum $T_1(x)$. As we will see, the main contribution to $F_1(x)\log x$ comes from this term.  Since the function $f$ is no longer assumed to be totally multiplicative, our key idea here is to make use of the properties of our class $\mathcal{ C}(k)$ to obtain for $\Re(s)>1$,  
\begin{align}
	\sum_{n=1}^\infty \frac{f(n)\log n}{n^s} = F(s) \left(-\frac{F'}{F}(s) \right) 
	&= \sum_{m =1}^\infty \frac{f(m)}{m^s} \sum_{d=1}^\infty \frac{\Lambda_F(d)}{d^s} 
	= \sum_{n =1}^\infty \frac{1}{n^s}\sum_{d|n}\Lambda_F(d) f(n/d). 
\end{align}
Comparing coefficients of on both sides, we obtain the identity 
\begin{equation} \label{eq:identity for f}
	f(n) \log n  = \sum_{d|n} \Lambda_F(d) f(n/d). 
\end{equation}
We notice that the function $f\equiv 1$ gives us back the identity $\log n =\sum_{d|n} \Lambda(d)$ used in \cite{MontgomeryVaughan2001}. Using \eqref{eq:identity for f}, we can write 
\[
T_1(x) = \sum_{n \le x} \frac{1}{n}  \sum_{d|n} \Lambda_F(d) f(n/d)  = \sum_{d \le x} \frac{ \Lambda_F(d)}{d} \sum_{m \le x/d} \frac{f(m)}{m}, 
\]
after interchanging summation. It follows that 
\begin{equation}\label{eq:bdT1}
	T_1(x) \ll \sum_{d \le x}  \frac{ |\Lambda_F(d)|}{d}  |F_1(x/d)|. 
\end{equation}
As $f \in \mathcal{ C}(k)$, it follows that $f(n) \ll d_k(n)$ for all $n$. 

Let $h =x/ (\log x)^k$ and $x-h \le v \le x$. Then we have 
\begin{align}
	T_1(x) -T_1(v) &= \sum_{v <n \le x} \frac{f(n) \log n}{n}
\ll \sum_{x-h \le n \le x } \frac{d_k(n) \log n}{n}	\\
	&\ll  \frac{\log x}{x}  \sum_{x-h \le n \le x } d_k(n)
\end{align}
as $n \asymp x $ in this range.  Using Theorem 2 of \cite{Shiu1979} to sum the $k$th divisor function in short intervals, we see that  the final sum in the above expression is $\ll h (\log x)^{k-1}$, giving 	$T_1(x) -T_1(v)  \ll 1$ due to the choice of $h$.  This gives us 
\begin{align}
	T_1(x) \ll 1 + \frac{1}{h}\int_{x-h}^x |T_1(v)|dv 
	&\ll 1+  \frac{1}{h}\int_{x-h}^x  \sum_{d \le v}  \frac{ |\Lambda_F(d)|}{d}  |F_1(v/d)| dv  \quad \text{(by \eqref{eq:bdT1})}
	\\
	&\ll  1+  \frac{1}{h}\sum_{d \le x}  \frac{ |\Lambda_F(d)|}{d} \int_{x-h}^x   |F_1(v/d)| dv, 
\end{align}
after interchanging the sum and integral, and noting that 
\[
\int_{\max{(d, x-h)} }^x   |F_1(v/d)| dv \le \int_{x-h}^x   |F_1(v/d)| dv. 
\]
As $F_1(u) =1$ for $1 \le u < 2$, it can be seen that the sum over $x/2 < d \le x-h$ is $\gg 1$, so that we may drop the  $O(1)$ term above. 
After a change of variable and interchanging the sum and integral once again, we get 
\begin{align}
	T_1(x) &\ll    \frac{1}{h} \sum_{d \le x}  { |\Lambda_F(d)|} \int_{\frac {x-h}{d} }^{x/d}   |F_1(u)| du 
	\ll 
	 \frac{1}{h}\int_{1}^x |F_1(u)| \bigg( \sum_{\frac{x-h}{u} \le d \le \frac{x}{u}}  { |\Lambda_F(d)|}\bigg) du. 
\end{align}
We now split the above integral as $\int_1^y + \int_y^x$, where $y =h/\log x = x/(\log x)^{k+1} $. 

Let us first consider $\int_{1}^y$, that is, the case $1 \le u \le y$. Then $x/u, h/u$ both tend to infinity as $x \rightarrow \infty$ and 
$ h/u = (x/u)(\log x)^{-k}$. 
Then we have, 
\begin{equation}
\sum_{\frac{x-h}{u} \le d \le \frac{x}{u}}  { |\Lambda_F(d)|}  \ll k \sum_{\frac{x}{u} -\frac{h}{u} \le d \le \frac{x}{u}}  \Lambda(d) \ll k \log \left(\frac{x}{u}\right)  \left( \pi\left(\frac{x}{u} \right) - \pi\left(\frac{x}{u}- \frac{h}{u} \right) \right).
\end{equation}
By Theorem 2 of \cite{MontgomeryVaughan1973}, the term in parenthesis above is less than
\[
 {2(h/u)}{(\log(h/u))^{-1}}.
\]
As 
\[
\log (x/u) = \log(h/u) + k\log \log x,
\]
 this gives 
\[
\sum_{\frac{x-h}{u} \le d \le \frac{x}{u}}  { |\Lambda_F(d)|} \ll \frac{kh}{u} \left( 1+ k \frac{\log\log x}{\log(h/u) }\right)
\ll k(k+1) h/u, 
\]
since $h/u \ge \log x$. 

Putting this back into our integral, we have obtained  
\begin{equation} \label{eq:int1}
\int_{1}^y  \ll k(k+1)h \int_1^y |F_1(u)| \frac{du}{u}. 
\end{equation}

Let us now turn to the remaining range $y \le u \le x$. In this range, as $\log u \asymp \log x$ and $x$ must satisfy the hypothesis \eqref{eq:condition on x}, we see that $F_1(u) \ll |F_1(x)|$. Using this along with interchange of summation, we have 
\begin{align}\label{eq:int2}
\int_{y}^x \ll k |F_1(x)| \int_y^x \sum_{\frac{x-h}{u} \le d \le \frac{x}{u}}  { \Lambda(d)}\, du 
&\ll k |F_1(x)| \sum_{d \le x/y} \Lambda(d) \int_{\frac{x-h}{d}}^{\frac{x}{d}}   du 
\\
&\ll kh|F_1(x)| \sum_{d\le  (\log x)^{k+1}} \frac{\Lambda(d)}{d}
\\
&\ll k(k+1) h |F_1(x)| \log \log x.  
\end{align}

Combining \eqref{eq:int1} and \eqref{eq:int2} gives 
\begin{equation} \label{eq:bound for T1}
T_1(x) \ll k(k+1) \left(   \int_1^x |F_1(u)| \frac{du}{u} +  |F_1(x)| \log \log x \right).   
\end{equation}

With this in hand, our eventual goal is to show that 
\begin{equation} \label{eq:T1 goal}
\int_e^x |F_1(u)| \frac{du}{u} \ll k \int_{1/\log x}^1 \frac{M_1(\a)}{\a} d\a.
\end{equation}
As a first step towards this, we will show that 
\begin{equation} \label{eq:T1 step1}
\int_e^x |F_1(u)| \log u \frac{du}{u} \ll k M_1\left( \frac{2}{\log x}\right) \log x. 
\end{equation}
As the Cauchy-Schwarz inequality gives
\begin{align}
\int_e^x |F_1(u)| \log u \frac{du}{u}\le \left( \int_e^x  \frac{du}{u} \right)^{1/2}
 \left( \int_e^x  |F_1(u)|^2 (\log u)^2 \frac{du}{u} \right)^{1/2},  
\end{align}
it suffices to show that 
\[
 \int_e^x  |F_1(u)|^2 (\log u)^2 \frac{du}{u} \ll k^2  M_1\left( \frac{2}{\log x}\right)^2 \log x. 
\]
Even more generally, it is enough to show that 
\begin{equation}\label{eq:T1step1general}
\int_e ^\infty |F_1(u)|^2 (\log u)^2 \frac{du}{u^{1+2\a} } 
\ll k^2 \alpha^{-1} M_1(\a)^2, 
\end{equation}
for $0 < \a \le 1$. 
To see this, let us observe that putting $\a= 2/(\log x)$ in the above expression gives 
\begin{align}
 k^2 M_1\left( \frac{2}{\log x}\right)^2 \log x \gg 
\int_e ^\infty |F_1(u)|^2 (\log u)^2 \frac{du}{u^{1+(4/\log x)} } &\gg
\int_e ^x |F_1(u)|^2 (\log u)^2  e^{-\frac{4\log u}{\log x}} \frac{du}{u}  \\
&\ge e^{-4} \int_e ^x |F_1(u)|^2 (\log u)^2  \frac{du}{u}, 
\end{align}
as required. 

We focus on proving the bound \eqref{eq:T1step1general}. 
Writing 
\begin{align}
F_1(u) \log u := \sum_{n \le u} \frac{f(n)}{n} \log u &= 
\sum_{n \le u} \frac{f(n) \log n }{n} + 
\sum_{n \le u} \frac{f(n) \log (u/n) }{n}
\\
&= S_1(u) + S_2(u)  \qquad \text{ (say)},  
\end{align}
we see that $
|F_1(u)|^2 (\log u)^2 \le  |S_1(u)|^2 + |S_2(u)|^2 + 
O\big( |S_1(u) ||S_2(u)| \big).$
However, as 
\[
S_1(u) S_2(u) := \sum_{n \le u} \frac{f(n) \log n }{n}  \sum_{m \le u} \frac{f(m) \log (u/m) }{m} 
\ll (\log u)^2 |F_1(u)|^2, 
\]
we obtain 
\[
|F_1(u)|^2 (\log u)^2 \ll  |S_1(u)|^2 + |S_2(u)|^2. 
\]
Using this followed by Plancherel's identity, we have  
\begin{align}
\int_e ^\infty |F_1(u)|^2 (\log u)^2 \frac{du}{u^{1+2\a} } 
&\ll \int_1 ^\infty \bigg| \sum_{n \le u}  \frac{f(n)  }{n}\log n \bigg|^2   \frac{du}{u^{1+2\a} } 
+ \int_1 ^\infty  \bigg| \sum_{n \le u} \frac{f(n)  }{n}\log (u/n) \bigg|^2  \frac{du}{u^{1+2\a} } 
\\
&\ll  \int_{-\infty}^{\infty} \bigg| \frac{F'(1+\a+ it)}{\a+it}  \bigg|^2  dt 
+  \int_{-\infty}^{\infty} \bigg|\frac{F(1+\a+it)}{(\a+it)^2}  \bigg|^2  dt.  
\end{align}
We treat the second integral above by partitioning the range into intervals $[m- \frac{1}{2}, m+\frac{1}{2}]$ of length $1$. This gives 
\begin{align}
    \int_{-\infty}^{\infty} \bigg|\frac{F(1+\a+it)}{(\a+it)^2}  \bigg|^2  dt
    &\ll \sum_{m=\infty} ^\infty \left( \max_{ t \in [m-1/2, m+1/2]  } \left| 
    \frac{F(1+\a+it) }{\a+it} \right|^2      \right) \int_{m-\frac{1}{2}}^{m+\frac{1}{2}} \left|(\a+it) \right|^{-2} dt 
 \\
    \label{eq:intg2 bound}
    &\ll \a^{-1} M_1(\a)^2 
\end{align}
It remains to show a similar bound for the first integral. Again, breaking the range into intervals of length $1$ and writing $F' = F \cdot F'/F$ gives 
\begin{align}
    \int_{-\infty}^{\infty} \bigg| \frac{F'(1+\a+ it)}{\a+it}  \bigg|^2  dt 
    &\ll \sum_{m=-\infty} ^\infty \left( \max_{ t \in [m-1/2, m+1/2]  } \left| 
    \frac{F(1+\a+it) }{\a+it} \right|^2      \right) \int_{m-\frac{1}{2}}^{m+\frac{1}{2}} \left|\frac{F'}{F}(1+\a+it) \right|^2 dt 
\\    
&\ll M_1(\a)^2 \sup_m \int_{m-\frac{1}{2}}^{m+\frac{1}{2}} \left|\frac{F'}{F}(1+\a+it) \right|^2 dt.  
\end{align}
To prove \eqref{eq:T1step1general}, it is thus enough to show that 
\[
\int_{T-1/2}^{T +1/2} \left|\frac{F'}{F}(1+\a+it) \right|^2 dt \ll \a^{-1},  
\]
uniformly for $0<\alpha \le 1$. 
To do this, we use (15) of \cite{MontgomeryVaughan2001}, which states that if $|a_n| \le b_n$ for all $n$, then 
\begin{equation} \label{eq:(15) of MV} 
\int_{T-U}^{T+U} \left| \sum_n a_n n^{-it} \right|^2 dt \le 3 \int_{-U}^U 
\left| \sum_n b_n n^{-it} \right|^2 dt. 
\end{equation} 
As $f \in \mathcal{ C}(k)$ implies that 
\[
\frac{F'}{F}(1+\a+it) = \sum_n \frac{\Lambda_F(n)}{n^{1+\a}} n^{-it}, 
\]
we may apply \eqref{eq:(15) of MV} to $a_n = \Lambda_F(n)/ n^{1+\a} $ and $b_n = k \Lambda(n) / n^{1+\a}$. This gives 
\begin{align}
\int_{T-1/2}^{T +1/2} \left|\frac{F'}{F}(1+\a+it) \right|^2 dt 
&\le 3k^2 \int_{-1/2}^{1/2} \left| \frac{\zeta'}{\zeta}(1+\a+it)  \right|^2 dt
\\
 &\ll k^2 \int_{-1/2}^{1/2} |\alpha+it|^{-2} dt \ll k^2  \alpha^{-1}. 
\end{align}
This shows that the integral 
\begin{equation} \label{eq:intg bound} 
\int_{-\infty}^{\infty} \bigg| \frac{F'(1+\a+ it)}{\a+it}  \bigg|^2  dt \ll k^2 \a^{-1} M_1(\a)^2
\end{equation}
We have thus shown that \eqref{eq:T1step1general} and hence \eqref{eq:T1 step1} holds. Recall that our aim was to obtain the bound \eqref{eq:T1 goal} for $ \int_e^x |F_1(u)| du/u$. 
Performing integration by parts on 
\[
\int_e^x  \frac{1}{\log u} \left( |F_1(u)| \frac{\log u}{u}   du \right)   
\]
gives 
\[
\int_e^x |F_1(u)| \frac{du}{u} \le 
\frac{1}{\log x}\int_e^x |F_1(t)| \frac{\log t}{t}   dt + 
\int_e^x \frac{1}{(\log u)^2}  \left( \int_e^u |F_1(t)| \frac{\log t}{t}   dt \right) \frac{du}{u}. 
\]
By \eqref{eq:T1 step1}, followed by the change of variable $\a= (\log u)^{-1}$, this is 
\[
\ll k M_1\left( \frac{2}{\log x}\right) +k \int_{1/\log x}^1  M_1(2 \alpha) \a^{-1} d\a. 
\]
Since $M_1(\a)$ is a decreasing function of $\a$ by definition, we can bound the above as 
\[
\ll k  \int_{1/\log x}^1  M_1( \alpha) \a^{-1} d\a, 
\]
 thereby proving \eqref{eq:T1 goal}. 
This completes our estimate for $T_1(x)$, given by 
\begin{equation}
\label{eq:T1finalbd}
T_1(x) \ll k(k+1) \left( k \int_{1/\log x}^1  M_1( \alpha) \a^{-1} d\a + |F_1(x)| \log \log x\right).  
\end{equation}

\subsection{Estimate for $T_2(x)$}
Recall that 
\[
T_2(x) \log x := \sum_{n \le x}\frac{f(n)}{n}\log n\log(x/n).
\]
We have for any $\a >0$, 
\[
T_2(x) \log x = \frac{1}{2\pi i } \int_{\a - i \infty}^{\a+i \infty} \frac{F'(s+1)}{s^2} x^s ds  \ll 
 \int_{-\infty}^{\infty} \frac{|F'(1+\a+it)|}{|\a+it|^2} x^{\a} dt. 
\]
Letting $1/\log x \le \alpha \le 2/\log x$, and then using the Cauchy-Schwarz inequality, we have 
\begin{align}
T_2(x) \log x &\ll  \int_{-\infty}^{\infty} \frac{|F'(1+\a+it)|}{|\a+it|^2} dt 
\\
&\ll 
\left( \int_{-\infty}^{\infty} \frac{1}{|\a+it|^2} dt \right)^{1/2}  
\left( \int_{-\infty}^{\infty} \frac{|F'(1+\a+it)|^2}{|\a+it|^2} dt \right)^{1/2}.   
\end{align}
By \eqref{eq:intg bound}, we obtain 
\[
T_2(x) \ll  \frac{ k M_1(\a)}{\a \log x}, 
\]
uniformly for  $1/\log x \le \alpha \le 2/\log x$. Hence, we can conclude that 
\begin{equation} \label{eq:T2finalbd}
T_2(x) \ll k \int_{1/\log x}^{2/\log x} \frac{M_1(\a)}{\a} d\a. 
\end{equation}

\subsection{Estimate for $T_3(x)$}
Similar to as done for $T_2$, we write 
\begin{align}
T_3(x) \log x &:= \sum_{n \le x} \frac{f(n)}{n} \big( \log(x/n)\big)^2
\\
&= \frac{1}{\pi i} \int_{\a - i\infty}^{\a+ i \infty } 
\frac{F(s+1)}{s^3} x^s ds, 
\end{align}
for $\a >0$.
Letting $1/\log x \le \alpha \le 2/\log x$, and then using the Cauchy-Schwarz inequality, we have 
\begin{align}
T_3(x) \log x &\ll  \int_{-\infty}^{\infty} \frac{|F(1+\a+it)|}{|\a+it|^3} dt 
\\
&\ll 
\left( \int_{-\infty}^{\infty} \frac{1}{|\a+it|^2} dt \right)^{1/2}  
\left( \int_{-\infty}^{\infty} \left| \frac{F(1+\a+it)}{(\a+it)^2} \right|^2 dt \right)^{1/2}.   
\end{align}
By \eqref{eq:intg2 bound}, we have 
\[
T_3(x) \ll \frac{ M_1(\a)}{\a \log x}, 
\]
uniformly for  $1/\log x \le \alpha \le 2/\log x$. As before, this allows us to conclude that 
\begin{equation} \label{eq:T3finalbd}
T_3(x) \ll \int_{1/\log x}^{2/\log x} \frac{M_1(\a)}{\a} d\a. 
\end{equation}

\subsection{A bound for $F_1(x)$}
Putting together \eqref{eq:T1finalbd}, \eqref{eq:T2finalbd} and \eqref{eq:T3finalbd}, we have the bound 
\[
F_1(x) \log x \ll  k(k+1) \left( k \int_{1/\log x}^1  M_1( \alpha) \a^{-1} d\a + |F_1(x)| \log \log x\right).
\]
As $k(k+1)\log \log x = o(\log x) $ for $x$ sufficiently large, we obtain 
\[
F_1(x) \ll \frac{k^2(k+1)}{\log x} \int_{1/\log x}^1 M_1( \alpha) \a^{-1} d\a, 
\]
thus completing the proof of Theorem \ref{thm:F1M1}.

\section{Mean Value Estimates: Proof of Theorem \ref{thm:F1estimate}} \label{sec7}
Before beginning with the proof of Theorem \ref{thm:F1estimate}, we set up the following lemmas. We note that by putting  $k=1$, we recover Lemmas 1 and 2 of \cite{MontgomeryVaughan2001}. 

\subsection{Preliminary Lemmas} 
\begin{lemma}\label{lem:MVlemma1}
	Let $f\in \mathcal{C}(k)$ and $F(s)$ be the associated Dirichlet series. Then for $1<\s_1\leq \s_2\leq 2$ we have 
	\begin{align}
	\left(\frac{\s_1-1}{\s_2-1}\right)^{k}	\ll\left\lvert\frac{F(\s_2)}{F(\s_1)}\right\rvert\ll\left(\frac{\s_2-1}{\s_1-1}\right)^{k}.
	\end{align}
\end{lemma}
\begin{proof}
By \eqref{eq:log of F}, we find that 
	\begin{align}
\left\lvert\frac{F(\s_2)}{F(\s_1)} \right\rvert \leq \exp \left( \left\lvert \sum_{n=1}^{\infty} \frac{\Lambda_F(n)}{\log n}\left(\frac{1}{n^{\s_1}}-\frac{1}{n^{\s_2}}\right) \right\rvert \right) 
&\leq \exp \left( k \sum_{n=1}^{\infty} \frac{\Lambda(n)}{\log n}\left(\frac{1}{n^{\s_1}}-\frac{1}{n^{\s_2}}\right)  \right)
\\  &=\left(\frac{\zeta(\s_1)}{\zeta(\s_2)}\right)^k\asymp\left(\frac{\s_2-1}{\s_1-1}\right)^k,
\end{align}
since 
\begin{equation}\label{eq:asymp zeta}
\zeta(\s) \asymp \frac{1}{\s -1} 
\end{equation}
for $1 < \s \le 2$.
This proves the upper bound. 	Similarly, one can show that
\begin{align}
\left\lvert\frac{F(\s_1)}{F(\s_2)}\right\rvert\ll\left(\frac{\s_2-1}{\s_1-1}\right)^k,
\end{align}
which  gives the required lower bound. 
\end{proof}

\begin{lemma}\label{lem:MVlemma2}
	Let $f\in \mathcal{C}(k)$ and $F(s)$ be the associated Dirichlet series. If $1<\s\leq 2$ and $|t|\leq 2$ then 
	\begin{align}
	\frac{F(\s+it)}{F(\s)}\ll_k \left(1+\frac{|t|}{\s-1}\right)^{4k/\pi}.
	\end{align}
	If $1 < \s\leq 2$ and $|t|\geq 2$ then 
	\begin{align}
	\frac{F(\s+it)}{F(\s)}\ll_k \left(\frac{\log|t|}{\s-1}\right)^{4k/\pi}.
	\end{align}
\end{lemma}
\begin{proof}
	As done in the proof of the previous lemma, by \eqref{eq:log of F}, we have 
\begin{align}
\left\lvert\frac{F(\s +it )}{F(\s)} \right\rvert &\leq \exp \left( k  \sum_{n=1}^{\infty} \frac{\Lambda(n)}{\log n}\left\lvert\frac{1}{n^{\s}}-\frac{1}{n^{\s+it}} \right\rvert  \right) 
\\
&= \exp \left( k  \sum_{n=1}^{\infty} \frac{\Lambda(n)}{n^\s(\log n)}\left\lvert \sin\left( \frac{t}{2} \log n\right) \right\rvert  \right). 
\label{fof}
\end{align}
Since the Fourier series of $|\sin x|$ is given by 
\begin{align}
|\sin x|=\sum_{m=-\infty}^{\infty}\frac{2}{\pi(1-4m^2)}e^{2i m x},
\label{eq:fourier sin}
\end{align}
we obtain
\begin{align}
\left\lvert\frac{F(\s+it)}{F(\s)}\right\rvert&\leq  
\exp\left(2k\sum_{m=-\infty}^{\infty}\frac{2}{\pi(1-4m^2)}\sum_{n=1}^{\infty}\frac{\Lambda(n)}{(\log n)n^{\s-imt}}\right)=\prod_{m=-\infty}^{\infty}(\zeta(\s-imt))^{\frac{4k}{\pi(1-4m^2)}}
\end{align}
Applying \eqref{eq:asymp zeta} for the $m=0$ term,  we find
\begin{align}
\left\lvert\frac{F(\s+it)}{F(\s)}\right\rvert&\ll \frac{1}{(\s-1)^{4k/\pi}}\prod_{\substack{m=-\infty\\m\neq 0}}^{\infty}(\zeta(\s-imt))^{\frac{4k}{\pi(1-4m^2)}}.\label{ff1}
\end{align}
Now, let $|t|\geq 2$. There exists a constant $C$ (cf. (3.11.18) of Titchmarsh \cite{titchmarsh})  such that 
\begin{align}
|\zeta(s)|\geq \frac{1}{C\log |t|}\label{zb}
\end{align}
uniformly for $\s>1$ and $|t|\geq 2$. 
Noticing that the exponents in the product above are negative for $m \ne 0$, combining \eqref{ff1} and \eqref{zb} gives 
\begin{align}
\left\lvert\frac{F(\s+it)}{F(\s)}\right\rvert&\ll \frac{1}{(\s-1)^{4k/\pi}}\prod_{m=1}^{\infty}(C\log (m|t|))^{\frac{8k}{\pi(4m^2-1)}}.
\end{align}
Finally employing the fact $\log(m|t|)\ll \log m\log |t|$ we deduce that
\begin{align}
\left\lvert\frac{F(\s+it)}{F(\s)}\right\rvert &\ll  \frac{1}{(\s-1)^{4k/\pi}}(\log |t|)^{\sum_{m=1}^{\infty}\frac{8k}{\pi(4m^2-1)}} \exp\left(\sum_{m=1}^{\infty}\frac{8k\log\log m}{\pi(4m^2-1)}\right).  
\end{align}
As the series inside the parenthesis is absolutely convergent, we have
\[\left\lvert\frac{F(\s+it)}{F(\s)}\right\rvert
 \ll_k \frac{1}{(\s-1)^{4k/\pi}}(\log |t|)^{\sum_{m=1}^{\infty}\frac{8k}{\pi(4m^2-1)}}.  
\]
Plugging  $x=0$ into  both sides of \eqref{eq:fourier sin} gives that the sum of all the Fourier coefficients is zero, that is
\begin{align}
\sum_{m=1}^{\infty}\frac{4}{\pi(1- 4m^2)}= -\frac{2}{\pi}.
\end{align}
Thus, we have obtained 
\begin{align}
\left\lvert\frac{F(\s+it)}{F(\s)}\right\rvert&\ll_k\left(\frac{\log|t|}{\s-1}\right)^{\frac{4k}{\pi}}, 
\end{align}
which is the required bound for the case $|t|\ge 2$. 

We now consider the remaining case $|t| \le 2$. 
Note that, for $1<\s\leq 2$ 
\begin{align}
\frac{F'(s)}{F(s)}=-\sum_{n=1}^{\infty}\frac{\Lambda_F(n)}{n^s}\ll k\sum_{n=1}^{\infty}\frac{\Lambda(n)}{n^{\s}}=-k\frac{\zeta'(\s)}{\zeta(\s)}\ll\frac{k}{\s-1}.\label{fpf}
\end{align}
Applying \eqref{fpf} to the Laurent series of $F(s)$ about $s=\s$, we see that for $0\leq t\leq \s-1$, we have  
\begin{align}
\frac{F(\s+it)}{F(\s)}\ll_k 1\label{e1},  
\end{align}
as required. 
Now consider $t > \s-1$. From \eqref{fof}, we can write
\begin{align}
\left\lvert\frac{F(\s+it)}{F(\s)}\right\rvert&\ll_k\exp\left(2k\sum_{p}\frac{1}{p^{\s}}\left\lvert\sin\left(\frac{t}{2}\log p\right)\right\rvert\right).\label{e2}
\end{align}
Let $1<\s\leq 2$ and $\s-1< t\leq 2$. If $p\leq e^{1/t}$ then we have the estimate
\begin{align}
\sum_{p\leq e^{1/t}}\frac{1}{p^{\s}}\left\lvert\sin\left(\frac{t}{2}\log p\right)\right\rvert&\ll t	\sum_{p\leq e^{1/t}}\frac{\log p}{p}\\
&\ll t\log(e^{1/t}) \ll 1.\label{e3}
\end{align}
For $p>e^{1/(\s-1)} $, using partial summation and the estimate $\pi(x) \ll x$ for the number of primes up to $x$, we have 
\begin{align}
\sum_{p>e^{1/(\s-1)}}\frac{1}{p^{\s}}\left\lvert\sin\left(\frac{t}{2}\log p\right)\right\rvert&\ll 	\sum_{p> e^{1/(\s-1)}}\frac{1}{p^{\s}}\\
&\ll\frac{\pi(e^{1/(\s-1)})}{e^{\s/(\s-1)}} +\s \int_{e^{1/(\s-1)}}^{\infty}\frac{\pi(t)}{t^{\s+1}}~dt\\
&\ll 1+  \int_{e^{1/(\s-1)}}^{\infty}\frac{1}{t^{\s}}~dt 
\ll 1.\label{e4}
\end{align}
In order to bound the remaining sum over primes $ e^{1/t} <  p \le e^{1/(\s-1)}$, define
$
g(u) = \dfrac{\left\lvert\sin\left(\frac{t}{2}\log u \right)\right\rvert}{u}. 
$
By the Fourier series \eqref{eq:fourier sin}, we see that 
\begin{align}
\int_0^y |\sin u| du &= \frac{2}{\pi}y + O\left(  \sum_{\substack{ m=1 }}^\infty \frac{1}{m^2}  \left|\int_0^y {\cos(2mu)} ~du \right|\right)
\\
&= \frac{2}{\pi}y + O\bigg( \sum_{\substack{ m=1 }}^\infty  \frac{1}{m^2} \left| \frac{\sin(2my) }{m} \right| \bigg) \\ &=  
 \frac{2}{\pi}y + O(1). 
\end{align}
After a suitable change of variable, this yields 
\begin{align}
I(y) := \int_1^y g(u) du =  \frac{2}{\pi} \log y +O(1). \label{I(y)}
\end{align}
{Let $\pi (u) $ denote the number of primes $p \le u$. By the prime number theorem, we have }
\begin{align}
	\pi(u)  &=\li(u) + O(u \exp(-c\sqrt{\log u}))=:\li(u)+E(u), 
\end{align}  
where
\begin{align}\label{li}
	 \li(u)=\int_2^u \frac{dv}{\log v}.
	\end{align}
we observe that 
\begin{align}
\sum_{p \le y} \frac{1}{p^{\s}}\left\lvert\sin\left(\frac{t}{2}\log p\right)\right\rvert  
&\le \sum_{p \le y} \frac{1}{p}\left\lvert\sin\left(\frac{t}{2}\log p\right)\right\rvert\\ 
&= \int_{2}^y g(u) ~d\pi(u)\\
&= \int_{2}^y g(u) ~d(\li(u))+\int_{2}^y g(u) ~d(E(u)).
\label{lastep}
\end{align} 
From \eqref{li} followed by integration by parts, we have
\begin{align}
	\int_{2}^y g(u) ~d(\li(u))=   \int_{2}^y \frac{g(u)}{\log u} ~du
	&= \frac{I(y)}{\log y}+ \int_{2}^y\frac{I(u)}{u (\log u)^2}=\frac{2}{\pi}\log\log y +O(1),\label{maint}
	\end{align}
	after using \eqref{I(y)}.
For the remaining term of \eqref{maint}, using the series \eqref{eq:fourier sin} once again, we obtain
\begin{align}
	\int_{2}^y g(u) ~d(E(u)) 
= \frac{2}{\pi} \int_2^y \frac{1}{u} dE(u) + O \left( \sum_{m =1}^\infty \frac{1}{m^2} \left|\int_2^y \frac{\cos(tm \log u)}{u} ~dE(u) \right| \right).  
	\label{errt}
\end{align}
For any differentiable function $h(u) \ll 1/u$, with $h'(u) \ll 1/u^2$, integration by parts gives 
\[
\int_2^y h(u) dE(u) = h(y) E(y) - \int_2^y h'(u) E(u) \ll 1. 
\]
Using this bound for the integrals involved in \eqref{errt} allows us to deduce that 
\[
	\int_{2}^y g(u) ~d(E(u)) = O(1). 
\] 
{With this, putting \eqref{maint} back into \eqref{lastep}, we have obtained }
\begin{align}
\sum_{p \le y} \frac{1}{p}\left\lvert\sin\left(\frac{t}{2}\log p\right)\right\rvert  
 = \frac{2}{\pi} \log\log y + O(1).
 \end{align}
This gives
\begin{align}
\sum_{ e^{1/t} <  p \le e^{1/(\s-1)}} \frac{1}{p^{\s}}\left\lvert\sin\left(\frac{t}{2}\log p\right)\right\rvert \le  \frac{2}{\pi} \log \left(\frac{t}{\s-1}\right)+O(1)
\label{e5}
\end{align}
Combining \eqref{e1}, \eqref{e2}, \eqref{e3}, \eqref{e4} and \eqref{e5},  we have
\begin{align}
\frac{F(\s+it)}{F(\s)}\ll_k \left(1+\frac{|t|}{\s-1}\right)^{4k/\pi}
\end{align}
for $|t|\leq 2$. This  completes the proof of the lemma. 	
\end{proof}

\subsection{Proof of Theorem \ref{thm:F1estimate}}
We now combine Lemmas \ref{lem:MVlemma1} and \ref{lem:MVlemma2} above to obtain a bound on the quantity $M_1(\a)$ defined in \eqref{eq:M1(alpha)}. The following lemma represents this key step.  

\begin{lemma}\label{lem:keylemma3}
Let $f\in \mathcal{C}(k)$, $F(s)$ be the associated Dirichlet series and $1+\frac{1}{\log x} \le \s\leq 2$. Then for $0<\a\leq 1$, we have 
\begin{align}
M_1(\a)\ll |F(\s)|\left(\frac{(\s-1)^{k}}{\a^{k+1}}+\frac{(\s-1)^{k(1-4/\pi)}}{\a^{k}}+\frac{\a^{k(1-4/\pi)}}{(\s-1)^{k}}\right).
\end{align}
\end{lemma}
\begin{proof} We suppose that $0<\a\leq \b\leq 1$ throughout this proof. 
Our first goal is to show that for $|t|\leq \frac{1}{2}$, we have 
\begin{align} \label{bound-t small}
\frac{F(1+\b +it)}{\b+it}  \ll |F(\s)| \left(\frac{(\s-1)^k}{\a^{k+1} }  + \frac{(\s-1)^{k(1-4/\pi)}  }{\a^{k} }  + \frac{\a^{k(1-4/\pi)}}{ (\s-1)^k} \right). 
\end{align}	
To prove this,  consider first the case $\s\leq \b+1$. Then from Lemma \ref{lem:MVlemma2} we have
\begin{align}
\frac{F(1+\b+it)}{F(1+\b)}\ll \left(1+\frac{|t|}{\b}\right)^{4k/\pi}\ll |\b+it|\b^{-\frac{4k}{\pi}} \left(\b+|t|\right)^{\frac{4k}{\pi}-1}\ll_k|\b+it|\b^{-\frac{4k}{\pi}}.\label{f01}
\end{align}
From this and Lemma \ref{lem:MVlemma1} we have 
\begin{align}
\frac{F(1+\b+it)}{\b+it}\ll |F(\s)|\b^{-\frac{4k}{\pi}}\left\lvert\frac{F(1+\b)}{F(\s)}\right\rvert\ll |F(\s)|\b^{-\frac{4k}{\pi}}\left(\frac{\b}{\s-1}\right)^k\ll |F(\s)|\frac{ \a^{k-\frac{4k}{\pi}} } {(\s-1)^{k}},\label{tsmall case1}
\end{align}
since $\a\leq \b$. Next we consider $\s\ge \b+1$. Then by Lemma \ref{lem:MVlemma1} we have 
\begin{align}
\frac{|F(1+\b+it)|}{|F(\s+it)|}\ll\left(\frac{\s-1}{\b}\right)^{k}\ll \frac{|\b+it|}{|\a+it|}\left(\frac{\s-1}{\b}\right)^{k}\label{fab},
\end{align}
for $\b\geq \a$.
From this and Lemma \ref{lem:MVlemma2} we find
\begin{align}
\frac{F(1+\b+it)}{\b+it}\ll\frac{|F(\s+it)|}{|\a+it|}\left(\frac{\s-1}{\b}\right)^{k}\ll \frac{|F(\s)|}{|\a+it|}\left(1+\frac{|t|}{\s-1}\right)^{4k/\pi}\left(\frac{\s-1}{\a}\right)^{k}.\label{fff01}
\end{align}
If $|t|\leq \s-1$, then the right hand side above is
\begin{align}
\ll_k |F(\s)|(\s-1)^k\a^{-1-k}\label{tsmall case21},
\end{align}
whereas for $\s-1\leq |t|\leq 1/2$, it is 
\begin{align}
\ll |F(\s)| 
 \frac{(\s-1+|t|)^{4k/\pi}}{|t|}   (\s-1)^{-4k/\pi}\left(\frac{\s-1}{\a}\right)^{k} 
 &\ll_k |F(\s)|\,|t|^{\frac{4k}{\pi}- 1}\frac{( \s-1)^{k(1- 4/\pi) }} { \a^{k}}
 \\
&\ll_k |F(\s)|\frac{( \s-1)^{k(1- 4/\pi) }} { \a^{k}}.\label{tsmall case22}
\end{align}
Combining the bounds \eqref{tsmall case1}, \eqref{tsmall case21} and \eqref{tsmall case22}, we see that \eqref{bound-t small} is proved. 

Next, we will  prove a similar bound for the case $|t-m|\leq \frac {1}{2}$, with $m$ a non-negative integer. Proceeding in a manner analogous to the previous argument, let $\s\leq \b+1$.  Then by Lemma \ref{lem:MVlemma2} and Lemma \ref{lem:MVlemma1} we deduce 
\begin{align}
F(1+\b+it)\ll F(1+\b)\left(\frac{\log 2|m|}{\b}\right)^{4k/\pi}&\ll |F(\s)|\left(\frac{\b}{\s-1}\right)^k\left(\frac{\log 2|m|}{\b}\right)^{4k/\pi}\\
&\ll |F(\s)|\frac{\a^{k-\frac{4k}{\pi}}}{(\s-1)^k}(\log 2|m|)^{4k/\pi}, \label{fbb01}
\end{align}
since $\a \le \b$. 
We now consider the case  $\s \ge \b+1$. Then by Lemmas \ref{lem:MVlemma1} and \ref{lem:MVlemma2}, we have 
\begin{align}
F(1+\b+it)\leq |F(\s+it)|\left(\frac{\s-1}{\b}\right)^k\ll|F(\s)|\left(\frac{\log 2|m|}{\s-1}\right)^{4k/\pi}\left(\frac{\s-1}{\a}\right)^k.\label{fbb02}
\end{align}
Combining the two bounds above, for $|t-m|\le 1/2$, $m \in \mathbb{Z}, m\ne 0$, we have established 
\begin{align}\label{bound-t big}
F(1+\b+it) \ll |F(\s)|(\log 2|m|)^{4k/\pi} \left(\frac{(\s-1)^{k(1-4/\pi)}}{ \a^k} + \frac{\a^{k(1-4/\pi)}}{ (\s-1)^k}  \right). 
\end{align}
Using \eqref{bound-t small} and \eqref{bound-t big} in the definition \eqref{eq:M1(alpha)} of $M_1(\a)$ allows us to complete the proof of the lemma. 
\end{proof}

Plugging in the bound for $M_1(\a)$ given by Lemma \ref{lem:keylemma3}, into Theorem \ref{thm:F1M1} gives us
\begin{align}
\frac{1}{\log x}\int_{1/\log x}^1 \frac{M_1(\a)}{\a} d\a \ll  
|F(\s)| \left( (\s-1)^k (\log x)^k + (\s-1)^{k(1-4/\pi)} (\log x)^{k-1} + \frac{(\log x)^{\frac{4k}{\pi}-k-1}}{(\s-1)^k} \right).
\end{align} 
We complete the proof of Theorem \ref{thm:F1estimate} by observing that since $\s-1 \gg 1/\log x$, one has 
\[
 (\s-1)^{k(1-4/\pi)} (\log x)^{k-1} \gg \frac{(\log x)^{\frac{4k}{\pi}-k-1}}{(\s-1)^k}.  
\] 

\section{Proof of Theorem \ref{thm:zerofreeF}} \label{sec8}
Let $F_1(x)$ be as defined in \eqref{eq:F1}. 
Using integration by parts followed by Theorem \ref{thm:F1estimate}, we find that 
\begin{align}
\sum_{n> N}\frac{f(n)}{n^{\s}}
&=-F_1(N)N^{1-\s}+(\s-1)\int_N^{\infty}\frac{F_1(t)}{t^{\s}}~dt\\
&\ll |F(\s)|(\s-1)^{k}\left((\s-1)^{-4k/\pi}+\log N\right)(\log N)^{k-1}N^{1-\s}\label{fN}
\end{align}
for $1+\frac{1}{\log N}\leq \s\leq 2$. We note that the partial sum 
\begin{align}
F_N(s):= \sum_{n \le N} \frac{f(n)}{n^s} = F(s)-\sum_{n> N}\frac{f(n)}{n^{s}}.\label{ff}
\end{align}
If we replace $f(n)$ by $f(n)n^{-it}$ in \eqref{fN} and apply the resulting bound into  \eqref{ff},  we see that 
\begin{align}
F_N(s)=F(s)(1+O((\log\log N)^{k-\frac{4k}{\pi}}))
\end{align}
uniformly for 
\begin{align}
\s\geq 1+\left(\frac{4k}{\pi}-1\right)\frac{\log\log N}{\log N}.\label{hp}
\end{align}
Recalling that $f\in \mathcal{ C}(k)$ implies that $F(s)\neq 0$ for $\s>1$,  this means that  $F_N(s)\neq 0$ in the half-plane given by \eqref{hp}.

\section*{Acknowledgments} 
The authors would like to thank 
Andrew Granville, Adam Harper, Oleksiy Klurman,  Dimitris Koukoulopoulos, M. Ram Murty and  G. Tenenbaum for very useful suggestions and comments. 
This work was initiated by the authors while attending the workshop on Recent developments in Analytic Number Theory at MSRI, Berkeley and was continued while they were postdoctoral fellows at Rice University and University of Waterloo respectively. The authors express their sincere gratitude to  these institutes for providing a stimulating  atmosphere conducive to research.


\end{document}